\documentclass[11pt,leqno,a4paper, english]{smfart}

\usepackage{amsmath,amsfonts,amssymb,amsthm,amscd}
\usepackage[english]{babel}
\usepackage[mathscr]{eucal}
\usepackage{graphicx}
\usepackage{appendix}
\usepackage{color}
\usepackage{dsfont}
\usepackage{pdfsync}
\usepackage{hyperref}
\allowdisplaybreaks

\numberwithin{equation}{section}

\newcommand{\N}{{\mathbb N}}
\newcommand{\R}{{\mathbb R}}

\newcommand{\RR}{{\mathbb R}^2}
\newcommand{\eps}{\varepsilon}

\renewcommand{\div}{\operatorname{div}}
\newcommand{\curl}{\operatorname{curl}}
\newcommand{\loc}{\operatorname{{loc}}}
\newcommand{\supp}{\operatorname{{supp}}}

\renewcommand{\leq}{\leqslant}
\renewcommand{\geq}{\geqslant}

\newtheorem{theorem}{Theorem}[section]
\newtheorem{proposition}[theorem]{Proposition}
\newtheorem{lemma}[theorem]{Lemma}
\newtheorem{corollary}[theorem]{Corollary}

\theoremstyle{definition}
\newtheorem{definition}[theorem]{Definition}

\theoremstyle{remark}
\newtheorem{remark}[theorem]{Remark}
\newtheorem*{remark*}{Remark}

\title{The vortex-wave system with gyroscopic effects}
\author[C. Lacave]{Christophe Lacave}
\address[C. Lacave]{CNRS-Universit\'e Grenoble Alpes \\ Institut Fourier UMR 5582\\ 100, rue des math\'ematiques, 38610 Gi\`eres, France} 
\email{christophe.lacave@univ-grenoble-alpes.fr}
 \author[E. Miot]{Evelyne Miot}
\address[E. Miot]{CNRS-Universit\'e Grenoble Alpes \\ Institut Fourier UMR 5582\\ 100, rue des math\'ematiques, 38610 Gi\`eres, France} \email{evelyne.miot@univ-grenoble-alpes.fr}


\date{\today}

\begin{abstract}
In this paper, we study the well-posedness for a coupled PDE/ODE system describing the interaction of several massive point vortices moving within a vorticity backgound in a 2D ideal incompressible fluid. The points are driven by the velocity induced by the background vorticity, by the other vortices, and by a Kutta-Joukowski-type lift force creating an additional gyroscopic effect. This system reduces to the so-called vortex-wave system, introduced by Marchioro and Pulvirenti \cite{MarPul1, MarPul}, when the point vortices are massless. 

On the one hand, we establish existence of a weak solution before the first collision. We show moreover that the background vorticity is 
transported by the flow associated to the total velocity field. 
On the other hand, we establish uniqueness in the case where the vorticity is initially constant in a neighborhood of the point vortices. When all the densities of the point vortices have the same sign, no collision occurs in finite time and our results are then global in time. Our proofs strongly rely on the definition of a suitable energy functional.
\end{abstract}

\begin{document}

\maketitle

\section{Introduction}

The purpose of this article is to investigate the well-posedness of the following PDE/ODE system:
\begin{equation}\label{syst:1}
\begin{cases}
\displaystyle \partial_t \omega+ \div (v \omega )=0,\\
\displaystyle v=u + \sum_{k=1}^N \frac{\gamma_{k}}{2\pi} \frac{(x-h_{k})^\perp}{|x-h_{k}|^2},\quad u = K\ast \omega, \quad K(x)=\frac{1}{2\pi} \frac{x^\perp}{|x|^2},\\
\displaystyle m_{k}\ddot{h}_{k}=\gamma_{k}\Big(\dot{h}_{k}- u(t,h_{k}) -\sum_{j\neq k} \frac{\gamma_{j}}{2\pi} \frac{(h_{k}-h_{j})^\perp}{|h_{k}-h_{j}|^2}\Big)^\perp \text{ for } k=1,\dots,N,
\end{cases}
\end{equation}
where 
$$\omega : [0,T]\times \RR\to \R, \quad h_{k}:[0,T]\to \RR \text{ for } k=1,\dots,N,$$ 
and where 
$$(m_{k},\gamma_{k}) \in \R^+_{*}\times \R \text{ for } k=1,\dots,N.$$
We supplement \eqref{syst:1} with the initial conditions
\begin{equation}\label{syst:2}
\begin{split}
&\omega(0,\cdot)=\omega_0\in L^\infty(\R^2) \text{ compactly supported in some }B(0,R_0),\\
& (h_{k},h'_{k})(0)=(h_{k,0},\ell_{k,0})\text{ for } k=1,\dots,N, \text{ with $h_{k,0}$ distincts}.
\end{split}
\end{equation}
\medskip

System \eqref{syst:1} for $N=1$ was derived by Glass, Lacave and Sueur \cite{GLS1} as an asymptotical system for the dynamics of a body immersed in a 2D perfect incompressible fluid, when the size of the body vanishes whereas the mass is assumed to be constant. The position of the body at time $t$ is represented by the position $h(t)$, the fluid is described by its divergence-free velocity $ u(t,x)$ and vorticity $\omega(t,x)=\curl u(t,x)$. Under suitable decay assumptions, the divergence free condition enables to recover the velocity explicitly in terms of the vorticity by the Biot-Savart law \cite{MarPul}: $ u=x^\perp/(2\pi |x|^2)\ast \omega$. The quantities $m$ and $\gamma$ are reminiscent of the mass of the body and of the circulation of the velocity around the body, respectively. The second order differential equation verified by $h$ means that the body is accelerated by a force that is orthogonal to the difference between the body speed and the fluid velocity at this point. This gyroscopic force is similar to the well-known Kutta-Joukowski-type lift force revealed in the case of a single body in an irrotational unbounded flow, see for instance \cite{Lamb,MarPul,Thomson}. Therefore, a byproduct of \cite{GLS1} is the existence of a global weak solution of \eqref{syst:1} when $N=1$.

In the case $N> 1$, it is not known whether the previous convergence result holds. The main goal of this paper is to establish the existence and the uniqueness (under an additional assumption on the initial data, see below) of solutions for any $N\geq 1$. In particular, we will prove that the trajectories of the points $h_k$ never collide if all the circulations $\gamma_{k}$ have the same sign. Such a result is important for example to justify the 2D spray inviscid model established by Moussa and Sueur \cite{MoussaSueur}, which was derived as a mean-field limit $N\to \infty$ of \eqref{syst:1}. We refer to that article for a comparison of the recent spray models introduced in the literature.

Before giving the precise statements of our theorems, we mention that \eqref{syst:1} reduces to the so-called vortex-wave system when setting $m_k=0$:
\begin{equation}\label{syst:wv}
\begin{cases}
\displaystyle \partial_t \omega+ \div (v \omega )=0,\\
\displaystyle v= u + \sum_{k=1}^N \frac{\gamma_{k}}{2\pi} \frac{(x-h_{k})^\perp}{|x-h_{k}|^2},\quad u = K\ast \omega, \quad K(x)=\frac{1}{2\pi} \frac{x^\perp}{|x|^2},\\
\displaystyle \dot{h}_{k} = u(t,h_{k}) +\sum_{j\neq k} \frac{\gamma_{j}}{2\pi} \frac{(h_{k}-h_{j})^\perp}{|h_{k}-h_{j}|^2} \text{ for } k=1,\dots,N.
\end{cases}
\end{equation}
And indeed, for $N=1$, Glass, Lacave and Sueur showed in \cite{GLS2} that the asymptotical dynamics of a small solid with vanishing mass evolving in a 2D incompressible fluid is governed by the vortex-wave system. The vortex-wave system was previously derived by Marchioro and Pulvirenti \cite{MarPul1, MarPul} to describe the interaction of a background vorticity $\omega$ interacting with one or several point vortices $h_k$ with circulations $\gamma_k$. Very recently, Nguyen and Nguyen have also justified the vortex-wave system as the inviscid limit of the Navier-Stokes equations \cite{Toan}. For System~\eqref{syst:wv}, existence of a weak solution (according to Definition~\ref{def:1} below) is proved up to the first collision time between the vortex trajectories. Concerning uniqueness, it is open in general, and it holds in the particular case when the vorticity $\omega$ is initially constant near the point vortices (namely the condition appearing in Theorem~\ref{theorem:main-several} below), as suggested in \cite{MarPul1,MarPul} and proved in \cite{LM,MiotThesis}. It is also proved in \cite{MarPul1} that if all the $\gamma_k$ have the same sign then no collision occurs in finite time therefore global existence holds. 

As for the spray model, these results are the first key to get a time of existence that is independent of $N$, in order to consider the homogenized limit (or mean-field limit) $N\to \infty$, for instance, used by Schochet \cite{Schochet} to justify the vortex method in $\R^2$. The main goal of this paper is to establish the corresponding existence and uniqueness results for the vortex-wave system with gyroscopic effects \eqref{syst:1}. From now on we will refer to the points $h_k$ in \eqref{syst:1} as ``massive'' point vortices.

\subsection*{Main results}

The first part of our analysis focuses on the existence issue for \eqref{syst:1}.

\begin{definition}\label{def:1}
 Let $T>0$. We say that $(\omega,\{h_k\}_{1\leq k\leq N})$ is a weak solution of \eqref{syst:1} on $[0,T]$, with initial data given by \eqref{syst:2}, if:
\begin{itemize}
\item $\omega\in L^\infty([0,T],L^1\cap L^\infty(\R^2))\cap C([0,T],L^\infty(\R^2)-w^\ast)$, $h_k \in C^2([0,T])$ for $k=1,\ldots,N$,
\item the PDE in \eqref{syst:1} is satisfied in the sense of distributions, and the ODEs in \eqref{syst:1} are satisfied in the classical sense. 
\end{itemize}
\end{definition}

\begin{theorem}\label{theorem:main-several-existence}
Let $\omega_0$ and $(\{h_{k,0}\},\{\ell_{k,0})\})$ be as in \eqref{syst:2}. There exists $T_{*}>0$ such that for any $T\in (0,T_{*})$, 
there exists a weak solution $(\omega,\{h_k\})$ to \eqref{syst:1} on $[0,T]$. Moreover, if we assume that $\gamma_k$ have the same sign for all $k=1,\ldots,N$, then $T_{*}=+\infty$. 
\end{theorem}

\begin{remark*}
The maximal time $T_{*}$ such that Theorem~\ref{theorem:main-several-existence} holds corresponds to the first collision between some of the massive points, and we will prove that no collision occurs in finite time if all the $\gamma_k$ have the same sign.
\end{remark*}

\begin{remark*}
If the initial vorticity $\omega_0$ was only assumed to be in $L^p_c(\RR)$ for some $p>2$, then one could still prove (global if all $\gamma_k$ have the same sign) existence of a weak solution to \eqref{syst:1} such that $\omega\in L^\infty( L^p)$. However in this case no uniqueness result is known.
\end{remark*} 

As already said, the same existence result is known to hold for the vortex-wave system \eqref{syst:wv}, see \cite{MarPul1}. The proof of Theorem~\ref{theorem:main-several-existence}, given in Section~\ref{sec:existence}, follows the same method as in \cite{MarPul1}, namely passing to the limit in an iterative scheme after establishing uniform estimates on the solution $(\omega_n, \{h_{k,n}\})$. To do so, we introduce a functional $\mathcal{H}_n$ in \eqref{def-Hn}. This functional is well-adapted to System~\eqref{syst:1} because it controls both the minimal distance between the vortex trajectories and the velocities; moreover, it can be shown that its time derivative is uniformly bounded. Except for the estimates we perform for this new functional $\mathcal{H}_n$, the proof of Theorem~\ref{theorem:main-several-existence} is quite straightforward and is not the main point of this paper.
 
Our next result is that any weak solution as in Theorem~\ref{theorem:main-several-existence} is actually transported by the regular Lagrangian flow relative to the total velocity field. We refer to the recent papers \cite{ambrosio, ambrosio-survey, crippa-delellis, bresiliens-miot} for the subsequent definition of regular Lagrangian flow:
 \begin{definition}\label{def:lagrangian-flow}
 Let $T>0$ and let $v\in L_{\text{loc}}^1([0,T]\times \R^2)$. We say that $X:[0,T]\times \R^2\times\R^2$ is a regular Lagrangian flow relative to $v$ if
 \begin{itemize}
 \item For a.e. $x\in \R^2$, the map $t\mapsto X(t,x)$ is an absolutely continuous solution to the ODE $\frac{d}{dt} X(t,x)=v(t,X(t,x))$ with $X(0,x)=x$, i.e. a continuous function verifying $X(t,x)=x+\int_{0}^t v(s,X(s,x)) \, ds$ for all $t\in [0,T]$;
 \item There exists a constant $L>0$ independent of $t$ such that
 $$\mathcal{L}^2 (X(t,\cdot)^{-1}(A)) \leq L\mathcal{L}^2 (A), \quad \forall t\in [0,T], \forall A \text{ Borel set of }\R^2,$$
 where $\mathcal{L}^2$ is the Lebesgue measure on $\R^2$.
 \end{itemize}

 \end{definition}
Such a definition is intended to generalize the classical notion of flow associated to smooth vector fields. It was proved by Ambrosio \cite{ambrosio} that such flow exists and is unique under BV space regularity for the vector field. In \cite{LM, bresiliens-miot}, 
a similar result was established for vector fields composed of a smooth part and of a part with a specific localized singularity that is explicit. In the present setting, where the total velocity field in \eqref{syst:1} contains singularities created by the point vortices, we will rely on those last results to establish the following general result.
\begin{theorem}\label{theorem:main-several-existence-lagrangian}
Let $\{h_{k}\}$ be any given maps belonging to $W^{2,\infty}([0,T];\RR)$ without collision: 
$$\min_{k\neq p} \min_{t\in [0,T]} |h_{k}(t)-h_{p}(t)|\geq\rho>0.$$
For $\omega_{0} \in L^\infty_{c}(\R^2)$, let $\omega$ be a weak solution on $[0,T]$ (in the sense of Definition~\ref{def:1}) to 
\begin{equation}\label{syst:1-2}
\begin{cases}
\displaystyle \partial_t \omega+ \div (v \omega )=0,\\
\displaystyle v= u + \sum_{k=1}^N \frac{\gamma_{k}}{2\pi} \frac{(x-h_{k})^\perp}{|x-h_{k}|^2},\quad u = K\ast \omega, \quad K(x)=\frac{1}{2\pi} \frac{x^\perp}{|x|^2},\\
\end{cases}
\end{equation}
such that $\omega(0,\cdot)=\omega_0$. Then, there exists a unique regular Lagrangian flow $X$ relative to the total velocity field $v$ and $\omega$ is transported by this flow: 
\[\omega(t,\cdot)=X(t,\cdot)_\#\omega_0.\] 

Moreover, the vorticity $\omega(t,\cdot)$ is compactly supported in $B(0,R_T)$ for all $t\in [0,T]$, where $R_T$ depends on $T$, on $\|h_k\|_{L^{\infty}([0,T])}$ and on the initial data.

Furthermore, we have the additional non collision information: 
$$ \text{for a.e. } x\in \RR,\quad X(t,x)\neq h_k(t), \quad\forall t\in [0,T], \quad \forall k=1,\ldots,N.$$

Finally, if we assume 
\begin{equation}\label{constant}
 \omega_0=\alpha_k\quad \text{on }B(h_{k}(0),\delta_0),\quad \forall k=1,\ldots,N
\end{equation}
for some $\alpha_k\in \R$ and $\delta_0>0$, there exists a positive $\delta$ depending only on 
$T$, $\delta_0$, $\|\omega_0\|_{L^\infty}$, $\|h_k\|_{W^{2,\infty}([0,T])}$ and $R_0$, such that
$$ \omega(t,\cdot)=\alpha_k\quad \text{on }B(h_k(t),\delta),\quad \forall t\in [0,T].$$
\end{theorem}

We emphasize that Theorem~\ref{theorem:main-several-existence-lagrangian} does not rely on the equation verified by the point vortices and thus it 
holds not only for \eqref{syst:1} but also for any system \eqref{syst:1-2}.

\medskip

We finally turn to the uniqueness issue of \eqref{syst:2}. 
\begin{theorem}\label{theorem:main-several}
Let $\omega_0$ and $(\{h_{k,0}\},\{\ell_{k,0})\})$ be as in \eqref{syst:2}. Assume moreover that 
$$ \omega_0=\alpha_k\quad \text{on }B(h_{k,0},\delta_0),\quad \forall k=1,\ldots,N$$
for some $\alpha_k\in \R$ and $\delta_0>0$. Then for any $T>0$, there exists at most one weak solution 
$(\omega,\{h_{k}\})$ to \eqref{syst:1} on $[0,T]$ with this initial condition.
\end{theorem}
The proof of Theorem~\ref{theorem:main-several} is a straightforward adaptation of the uniqueness proof given for the vortex-wave system in \cite{LM} when the vorticity is constant for all time in the neighborhood of the point vortices.
Hence the main difficulty in order to prove uniqueness under Assumption~\eqref{constant} is to prove the last point of Theorem~\ref{theorem:main-several-existence-lagrangian}.

\medskip

Theorem~\ref{theorem:main-several} together with Theorem~\ref{theorem:main-several-existence} thus implies global existence and uniqueness if all the $\gamma_{k}$ have the same sign, and existence and uniqueness up to the first collision otherwise.

\medskip

The plan of this paper is as follows. In the next section, we prove Theorem~\ref{theorem:main-several-existence} 
after collecting a few well-known properties. 
Then in Section~\ref{sec:lagrangian} we establish Theorem~\ref{theorem:main-several-existence-lagrangian}. Finally, in Section~\ref{sec:final-proof} we show how it implies Theorem~\ref{theorem:main-several} by adapting the arguments of \cite{LM, MiotThesis}. For simplicity we focus for this on the case of one point, but the case of $N\geq 1$ points is similar.
The last section is devoted to some additional properties satisfied by solutions of System~\eqref{syst:1}.

\medskip

With respect to the above-mentioned previous works, the main novelty for the proofs here is the use of a new local energy functional
\begin{equation}\label{defi:energy-intro}
F_k(t)=\sum_{j=1}^N \frac{\gamma_j}{2\pi} \ln |X(t,x)-h_j(t)|+\varphi(t,X(t,x))+\langle X(t,x),\dot{h}_k^\perp(t)\rangle,
\end{equation}defined as long as $X(t,x)\neq h_j(t)$, where $\varphi$ is the stream function associated to $u$ (namely $u=\nabla^\perp \varphi$, see \eqref{def:stream}). It turns out that the two last terms in the definition \eqref{defi:energy-intro} are uniformly bounded. 
Hence controlling the distances between the fluid particles and the massive point vortices 
(thus controlling the behavior of $\omega(t,\cdot)$ near those points) is made possible by proving that $F_k(t)$ is bounded. 
In the case of one point vortex, the following formal computation on the derivative of $F(t):=F_1(t)$ shows that the most singular terms cancel, which motivates our definition \eqref{defi:energy-intro}\footnote{We set $X=X(t,x)$ and $h=h(t)$ for more clarity.}:
\begin{equation*}
\begin{split}
F'=&\left\langle \frac{\gamma}{2\pi}\frac{X-h}{|X-h|^2},u(t,X)+\gamma K(X-h)-\dot{h}\right\rangle \\
&+\partial_t \varphi(t,X)+\langle \dot{X},\nabla\varphi(t,X)\rangle+\left\langle \dot{X}, \dot{h}^\perp \rangle+\langle X,\ddot{h}^\perp\right\rangle\\
=&\left\langle \frac{\gamma}{2\pi}\frac{X-h}{|X-h|^2},u(t,X)-\dot{h}\right\rangle +\partial_t \varphi(t,X)-\langle \dot{X},u^\perp(t,X)\rangle+\langle \dot{X}, \dot{h}^\perp \rangle+\langle X,\ddot{h}^\perp\rangle.
\end{split}
\end{equation*}
Since 
\begin{equation*}
\frac{\gamma}{2\pi}\frac{X-h}{|X-h|^2}=-\dot{X}^\perp+u^\perp(t,X),
\end{equation*}
we observe that the singular terms in the previous expression actually cancel. Finally, we get 
\begin{equation*}
F'(t)= -\langle u^\perp(t,X),\dot{h}\rangle +\partial_t \varphi(t,X)+\langle X,\ddot{h}^\perp\rangle.
\end{equation*} 
Thus it only remains to notice that this expression only involves bounded terms so that $|F'(t)|\leq C$ on $[0,T]$, as wanted. The rigorous proof of this bound for several points will be established in Proposition~\ref{prop:Fi}.

\medskip

\noindent
\textbf{Notations. } From now on $C$ will refer to a constant depending only on 
$T$, on $\rho$, on $\|h_k\|_{W^{2,\infty}([0,T])}$, and on the initial data ($R_0$, $m_k$, $\gamma_k$, $h_{k,0}$, $\ell_{k,0}$ and $\|\omega_0\|_{L^\infty}$), but not on $\delta_0$. It will possibly changing value from one line to another.

\medskip

\noindent
{\bf Acknowledgements.} The authors are partially supported by the Agence Nationale de la Recherche, Project SINGFLOWS, grant ANR-18-CE40-0027-01. C.L. was also partially supported by the CNRS, program Tellus, and the ANR project IFSMACS, grant ANR-15-CE40-0010. E.M. acknowledges French ANR project INFAMIE ANR-15-CE40-01.
Both authors thank warmly Olivier Glass and Franck Sueur for interesting discussions. They also thank warmly the anonymous referee for suggesting a simplification of the proof 
of Proposition~\ref{prop:ineg:flot}, which appears in the present paper.

\section{Proof of Theorem~\ref{theorem:main-several-existence}}

\label{sec:existence}

\subsection{Some general regularity properties}

We start with the following well-known property, see \cite[Appendix 2.3]{MarPul} for instance.
\begin{proposition}\label{prop:reg-u}
Let $\Omega\in L^\infty([0,T],L^1\cap L^\infty( \RR))$. Let $U=K\ast \Omega$. Then we have 
$$\|U\|_{L^\infty}\leq C\|\Omega\|_{L^\infty(L^1)}^{1/2}\|\Omega\|_{L^\infty}^{1/2}.$$ Moreover, $U$ is log-Lipschitz uniformly in time:
\begin{equation*}
\|U(\cdot,x)-U(\cdot,y)\|_{L^\infty(0,T)}\leq C(\|\Omega\|_{L^\infty(L^1\cap L^\infty)}) |x-y|(1+|\ln|x-y||),\quad \forall (x,y)\in \RR\times \RR.
\end{equation*}
\end{proposition}
We also have the Calder\'on-Zygmund inequality \cite[Chapter II, Theorem 3]{stein}
\begin{proposition}\label{prop:calderon} There exists $C$ such that for all $p\geq 2$
\begin{equation*}
\|\nabla U(t,\cdot)\|_{L^p}\leq Cp \|\Omega(t,\cdot)\|_{L^p}\quad \text{for all } t\in [0,T].
\end{equation*}
\end{proposition}
In particular, it follows that any such velocity field satisfies
\begin{equation}
\label{regu:field}
U\in L^\infty([0,T]\times \RR))\cap L^\infty([0,T],W^{1,1}_{\text{loc}}(\RR)),\quad \text{div}(U)=0.
\end{equation}

\subsection{Some basic properties for weak solutions of \eqref{syst:1}-\eqref{syst:2}}
\label{subsec:basic}

In all this paragraph, $(\omega, \{h_k\})$ denotes a weak solution of \eqref{syst:1}-\eqref{syst:2} on $[0,T]$, so that in particular $u$ satisfies Proposition~\ref{prop:reg-u} and the regularity property \eqref{regu:field}. We assume \emph{moreover} that
$\omega(t,\cdot)$ is compactly supported in some $B(0,R)$ for all $t\in [0,T]$.

\medskip

We introduce the stream function
\begin{equation}\label{def:stream}
\varphi(t,x)=\frac{1}{2\pi}\int_{\R^2}\ln |x-y|\omega(t,y)\,dy,
\end{equation}
so that
\begin{equation*}
u(t,x)=\nabla^\perp \varphi(t,x).
\end{equation*}
For the subsequent computations, in order to make the arguments rigorous, we introduce a regularized version of the stream function: for $\eps>0$ and 
$\ln_\eps$ a smooth function coinciding with $\ln$ on $[\eps,+\infty)$ and satisfying 
$|\ln_\eps'(r)|\leq C/\eps$ for all $r>0$, we set 
\begin{equation}\label{def:stram-2}
\varphi_\eps(t,x)=\frac{1}{2\pi}\int_{\RR} \ln_\eps |x-y|\omega(t,y)\,dy.
\end{equation}
Note that by assumption on the support of $\omega(t,\cdot)$ the following estimate holds for $\varphi_\eps$:
\begin{equation} \label{est:phi}
|\varphi_\eps(t,x)|\leq C\ln (2+|x|),
\end{equation}
with $C$ also independent of $\varepsilon$.
 
The following bound will be useful in order to establish a bound on the local energies in Proposition ~\ref{prop:Fi}:
\begin{proposition}\label{prop:ineq:partialt-several}
There exists $C_R$ depending only on $R$, $\|\omega\|_{L^\infty}$, $\|h_k\|_{L^\infty}$, and the initial data, such that
\begin{equation*}
\| \partial_t \varphi_\eps \|_{L^\infty} \leq C_R.
 \end{equation*}
\end{proposition}

\begin{proof}
Using the weak formulation for $\omega$ in \eqref{syst:1}, we have
\begin{multline*}
 \partial_t \varphi_\eps(t,x)\\
 =\frac{1}{2\pi}\int_{\RR} \ln_\eps'(|x-y|)\frac{y-x}{|y-x|}\cdot\left(u(t,y)+\sum_{k=1}^N \gamma_k K(y-h_k(t))\right)\omega(t,y)\,dy,
 \end{multline*}
 therefore
 \begin{align*}
| \partial_t \varphi_\eps(t,x)|
\leq& C\|u\|_{L^\infty} \int_{\RR} \frac{|\omega(t,y)|}{|x-y|}\,dy\\
&+C \sum_{k=1}^N \int_{\RR} \left|\frac{(x-y)\cdot(y-h_k(t))^\perp}{|x-y|^2|y-h_k(t)|^2} \omega(t,y)\right|\,dy.
 \end{align*}
By the estimates (1.39) to (1.43) in \cite{Marchioro}, there exists a constant $C$ depending only on $R$, on $\|\omega \|_{L^\infty}$, and on $\|h_k\|_{L^\infty}$, such that
\begin{equation*}\label{ineq:partialt}
| \partial_t \varphi_\eps(t,x)|\leq C_R.
 \end{equation*}
The conclusion follows.
\end{proof}
In the previous computation, we needed the smoothness of $\ln_{\varepsilon}$ in order to use the weak formulation for $\omega$. This explains why we have to replace $\varphi$ in the definition of $F_{k}$ \eqref{defi:energy-intro} by $\varphi_{\varepsilon}$ (see \eqref{def-Fk}) when we compute the derivative.

In the following proposition we state that $\nabla\varphi_{\varepsilon}$ approaches uniformly $\nabla\varphi$.

\begin{proposition}\label{prop:reste-vitesse}We have
 \begin{equation*}\nabla^\perp \varphi_\eps=u+R_\eps, 
 \end{equation*}
where $R_\eps$ satisfies
\begin{equation*}
\|R_\eps \|_{L^\infty}\leq \eps C\|\omega\|_{L^\infty}.
\end{equation*}
\end{proposition}
\begin{proof}
We have
 \begin{align*}
 \nabla^\perp \varphi_\eps(t,x)&=u(t,x)+\int_{\RR}\frac{(x-y)^\perp}{|x-y|^2}\left( |x-y|\ln'_\eps|x-y|-1\right)\omega(t,y)\,dy\\
 &=u(t,x)+R_\eps(t,x), 
 \end{align*}
where
\begin{equation*}
|R_\eps(t,x)|\leq C\int_{|y-x|\leq \eps}\frac{1}{|x-y|}|\omega(t,y)|\,dy\leq C\|\omega(t,\cdot)\|_{L^\infty}\eps.
\end{equation*}
\end{proof}

\subsection{Proof of Theorem~\ref{theorem:main-several-existence} }

The proof of Theorem~\ref{theorem:main-several-existence} is divided into two steps.
\medskip

\textbf{Step 1: iterative scheme.}
Let $\rho\in (0,\min_{k\neq p} |h_{k,0}-h_{p,0}|)$ which will be fixed later. We consider the following iterative scheme: for $n\in \mathbb{N}^\ast$, given 
$$
\omega_{n-1}\in L^\infty([0,T_{n-1}],L^1\cap L^\infty(\RR)),
$$
and given $N$ trajectories $h_{k,n-1}$ in $C^2([0,T_{n-1}])$ such that
$$\min_{t\in [0,T_{n-1}]}\min_{k\neq p} |h_{k,n-1}(t)-h_{p,n-1}(t)|>0,$$
for some $T_{n-1}>0$, we set 
$$u_{n-1}=K\ast \omega_{n-1},$$ 
having in mind to solve the linear PDE
\begin{equation}\label{syst:3}\begin{cases}
\displaystyle \partial_t \omega_{n}+\div \Bigg[\Big(u_{n-1}+\sum_{k=1}^N\frac{\gamma_k}{2\pi} \frac{(x-h_{k,n-1}(t))^\perp}{|x-h_{k,n-1}(t)|^2}\Big) \omega_{n}\Bigg]=0\\
\omega_{n}(0)=\omega_0,\end{cases}\end{equation}
and the non linear system of ODEs: for $k=1,\ldots,N$, 
\begin{equation}\label{syst:3-ODE}
\left\{
\begin{aligned}
& \displaystyle \ddot{h}_{k,n}(t)=\frac{\gamma_k}{m_{k}}\Big(\dot{h}_{k,n}(t)-u_{n-1}(t,h_{k,n}(t))\\
&\hspace{3cm}-\sum_{j\neq k}\frac{\gamma_j}{2\pi}\frac{(h_{k,n}(t)-h_{j,n}(t))^{\perp}}{|h_{k,n}(t)-h_{j,n}(t)|^2}\Big)^\perp\\
&(h_{k,n}(0),\dot{h}_{k,n}(0))=(h_{k,0},\ell_{k,0}),
\end{aligned}\right.
\end{equation}
on $[0,T_{n}]$, where $T_{n}\in (0,T_{n-1}]$ will be chosen such that
\begin{equation}\label{syst:3-dist}
\min_{t\in [0,T_{n}]} \min_{k\neq p} |h_{k,n}(t)-h_{p,n}(t)|\geq \rho. 
\end{equation}

For $n=0$ we take $\omega_0$ and $(h_{k,0},\ell_{k,0})$ as data (with $T_{0}=+\infty$).

\begin{proposition} \label{prop:n}
For all $n\in \mathbb{N}$, there exist $T_{n}\in (0,T_{n-1}]$ and a unique weak solution $\omega_{n}$ to \eqref{syst:3} and $\{h_{k,n}\}$ to \eqref{syst:3-ODE} on $[0,T_{n}]$ such that \eqref{syst:3-dist} is satisfied.

Moreover, $$\|\omega_{n}(t,\cdot ) \|_{L^1\cap L^\infty}\leq \|\omega_{0}\|_{L^1\cap L^\infty} \quad \forall t\in [0,T_{n}]$$ and there exists $\widetilde{T}$ depending only on $\rho$, $h_{k,0},\ell_{k,0}$, $R_0$ and $\|\omega_{0}\|_{L^\infty}$ such that $T_{n}\geq \widetilde{T}$ for all $n$.

Finally, if all the $\gamma_k$ have the same sign, then for any $T>0$, one can choose $\rho$ depending on $T$ (and on $h_{k,0},\ell_{k,0}$ and $\|\omega_{0}\|_{L^\infty}$, $R_0$) such that $T_n= T$ for all $n\in \mathbb{N}$.
\end{proposition}

\begin{proof} Given $(\omega_{n-1},\{h_{k,n-1}\})$ satisfying the bound of Proposition~\ref{prop:n}, we solve the \emph{linear} transport equation \eqref{syst:3} with initial data $\omega_0$ and velocity field given by
$$v_{n-1}(t,x)=u_{n-1}(t,x)+\sum_{k=1}^N\frac{\gamma_k}{2\pi} \frac{(x-h_{k,n-1}(t))^\perp}{|x-h_{k,n-1}(t)|^2}.$$
The existence of such a weak solution $\omega_{n}\in L^\infty([0,T_{n-1}],L^1\cap L^\infty(\RR))$ follows from classical arguments on 
linear transport equation. For the uniqueness issue, we refer to Lemma~\ref{appendix:1} (derived from \cite[Chapter 1]{MiotThesis}), 
which proves that any field $v_{n-1}$ given as above, with $u_{n-1}$ satisfying the regularity property \eqref{regu:field} 
and with the maps $h_{k,n-1}$ Lipschitz continuous and not intersecting on $[0,T_{n-1}]$, 
has the renormalization property (see \cite[Definition 1.5]{DeLellis} for the definition of renormalization). 
By the usual arguments for linear transport equations, see \cite{dip-lions}, uniqueness therefore holds in $L^\infty([0,T_{n-1}], L^1\cap L^\infty(\RR))$ 
for the linear transport equation associated to $v_{n-1}$. 

 Moreover, it follows from Corollary~\ref{coro:appendix} in the Appendix that the norms $\|\omega_{n}(t,\cdot)\|_{L^p}$ are constant in time 
 for all $p$, therefore we get the desired bound for $\|\omega_{n}\|_{L^1\cap L^\infty}$. 
 Recalling Proposition~\ref{prop:reg-u}, it follows that $\|u_n\|_{L^\infty}\leq C.$ Furthermore, the weak time continuity for $\omega_n$ established in \cite[Proposition 4.1]{LM} (see also \cite{MiotThesis}) implies that $u_n$ is uniformly continuous in space-time.

Next, in view of the almost-Lipschitz property and the time regularity for $u_{n-1}$, 	Osgood's lemma ensures that there exists a unique solution 
$\{h_{k,n}\}$ to \eqref{syst:3-ODE} on some maximal open interval $I_n\subset [0,T_{n-1}]$ such that
\[
\min_{k\neq p} |h_{k,n}(t)-h_{p,n}(t)| > 0, \quad \forall t\in I_n.
 \]
We consider then $T_n\leq T_{n-1}$ such that $[0,T_n)\subset I_n$ and $T_n$ is the largest time for which \eqref{syst:3-dist} holds:
\[
\min_{k\neq p} |h_{k,n}(t)-h_{p,n}(t)| > \rho,\quad \forall t\in [0,T_n).
 \]
Taking the scalar product of \eqref{syst:3-ODE} with $\dot{h}_{k,n}(t)$ and using Proposition~\ref{prop:reg-u} and the lower bound \eqref{syst:3-dist}, we get on $[0,T_{n}]$:
\[
m_{k} \frac{d |\dot{h}_{k,n}(t)|^2}{dt} \leq C |\dot{h}_{k,n}(t)|\leq C |\dot{h}_{k,n}(t)|^2+1
\]
hence we deduce by Gronwall that 
\begin{equation}\label{ineg:d-31}
|\dot{h}_{k,n}(t)|\leq C \quad \text{on } [0,\min (T_{n},1)], 
\end{equation}
(where we emphasize that $C$ depends on $\rho$), so 
\[
T_{n}\geq \widetilde{T}:= \min\Big(1, \frac{\min_{k\neq p} |h_{k,0}-h_{p,0}| -\rho}{2C}\Big).
\]

It remains to study the case where all $\gamma_k$ have the same sign (say positive), where we have to derive an inequality like \eqref{ineg:d-31} which is independent of $\rho$. We fix $T>0$ and we assume that $T_{n-1}=T$. We want to show that $T_n=T$.
In the sequel of this proof, $C$ depends only on the initial data and $T$. We seek for a uniform lower bound for the distances $|h_{k,n}-h_{p,n}|$ and for a uniform upper bound for $|\dot{h}_{k,n}|$ on $[0,T_n)$. To this aim, we introduce the quantity
\begin{equation}\label{def-Hn}
\mathcal{H}_n(t)=\sum_{j\neq k} \frac{\gamma_j \gamma_k}{2\pi} \ln |h_{j,n}(t)-h_{k,n}(t)|-\sum_{k=1}^N m_k |\dot{h}_{k,n}(t)|^2,
\end{equation}
defined on $[0,T_n]$. As we shall see below, bounding $|\mathcal{H}_n|$ uniformly with respect to $n$ allows to obtain the desired bounds on $|\dot{h}_{k,n}|$ and on $|h_{k,n}-h_{p,n}|$. In order to obtain a suitable estimate on $|\mathcal{H}_n|$, we compute the time derivative:
\begin{equation*}
\begin{split}
\dot{\mathcal{H}}_n&=\sum_{j\neq k}\frac{\gamma_j \gamma_k}{2\pi}
(\dot{h}_{k,n}- \dot{h}_{j,n})\cdot \frac{h_{k,n}-h_{j,n}}{|h_{k,n}-h_{j,n}|^2}-2 \sum_{k=1}^N m_k \dot{h}_{k,n}\cdot \ddot{h}_{k,n}\\
&= 2 \sum_{k=1}^N \dot{h}_{k,n}\cdot \left(\gamma_k \sum_{j\neq k} \frac{\gamma_j}{2\pi} \frac{h_{k,n}-h_{j,n}}{|h_{k,n}-h_{j,n}|^2}-m_k \ddot{h}_{k,n}\right),
\end{split}
\end{equation*}
where we have exchanged $j$ and $k$ in order to pass from the first line to the second one. Thus by \eqref{syst:3-ODE}, it only remains:
\begin{equation*}
\dot{\mathcal{H}}_n(t)=2\sum_{k=1}^N \gamma_k \dot{h}_{k,n}(t)\cdot u_{n-1}(t,h_{k,n}(t))^\perp.
\end{equation*}
Using the bound $\|u_{n-1}\|_{\infty}\leq C$ we get
\begin{equation}\label{ineg:H-1}
|\dot{\mathcal{H}}_n(t)| \leq C \sum_{k=1}^N \gamma_k |\dot{h}_{k,n}(t)|.
\end{equation}
On the other hand, we notice that for all $k$, for all $t\in [0,T_n]$, using that $\ln |x-y|\leq |x|+|y|$ we have
\begin{align*}
|\dot{h}_{k,n}(t)|\leq& C |{\mathcal{H}}_n(t)|^{1/2}+ C\left( \sum_{j=1}^N \gamma_j |h_{j,n}(t)|\right)^{1/2}\\
\leq& C |{\mathcal{H}}_n(t)|^{1/2}
+ C\left( \sum_{j=1}^N \gamma_j |h_{j,n}(0)|+\sum_{j=1}^N \gamma_j \int_0^t |\dot{h}_{j,n}(\tau)|\,d\tau\right)^{1/2}\\
\leq& C |{\mathcal{H}}_n(t)|^{1/2}
+ C\left( \sum_{j=1}^N \gamma_j |h_{j,0}|+\max_k \max_{ [0,T_n]} |\dot{h}_{k,n}|\right)^{1/2},
\end{align*}
hence 
\begin{equation*}
\begin{split}
\max_k \max_{ [0,T_n]} |\dot{h}_{k,n}|\leq &C \max_{ [0,T_n]} |{\mathcal{H}}_n|^{1/2}+C\sqrt{1+\max_k \max_{[0,T_n]} |\dot{h}_{k,n}|}\\
\leq&C \max_{ [0,T_n]} |{\mathcal{H}}_n|^{1/2}+C+ \frac{\max_k \max_{[0,T_n]} |\dot{h}_{k,n}|}{2},
\end{split}
\end{equation*}where we have used that $C\sqrt{1+a}\leq C\sqrt{2}+C^2+a/2$ for $a>0$.
Therefore
\begin{equation}\label{ineg:H-2}
\begin{split}
\max_k \max_{ [0,T_n]} |\dot{h}_{k,n}|\leq C \max_{ [0,T_n]} |{\mathcal{H}}_n|^{1/2}+C.
\end{split}
\end{equation}
Inserting \eqref{ineg:H-2} in \eqref{ineg:H-1} we also obtain 
\begin{align*}
\max_{[0,T_n]}|\dot{\mathcal{H}}_n|
\leq& C \left(1+\max_{ [0,T_n]}|{\mathcal{H}}_n|^{1/2}\right),
\end{align*}
therefore we get
\[
\max_{[0,T_n]}|{\mathcal{H}}_n|
\leq C\quad \text{and} \quad \max_{[0,T_n]}|\dot{\mathcal{H}}_n|
\leq C.
\]
Coming back to \eqref{ineg:H-2}, it follows that
\begin{equation}\label{ineg:d-2}
\max_k \max_{ [0,T_n]} |\dot{h}_{k,n}|\leq C ,
\end{equation}
so that
\begin{equation*}
\max_k \max_{ [0,T_n]} |{h}_{k,n}|\leq C .
\end{equation*}
Finally, by the definition of $\mathcal{H}_n(t)$ and by the previous bounds, using again that $\ln |x-y|\leq |x|+|y|$ we have for all $j\neq k$:
\begin{align*}
\gamma_j \gamma_k\ln |h_{j,n}(t)-h_{k,n}(t)|&\geq -2\pi |\mathcal{H}_n(t)| 
-\sum_{p,\ell =1}^N \gamma_p \gamma_\ell \left(|h_{p,n}(t)|+ |h_{\ell,n}(t)|\right)\\
&\geq -C,
\end{align*}
which means that there exists $\rho>0$ depending only on $T$ and the initial data such that
\begin{equation*}
\min_{j\neq k} |h_{j,n}(t)-h_{k,n}(t)|> \rho>0, \quad \forall t\in [0,\min (T_{n},T)).
\end{equation*}
Choosing this $\rho$ from the beginning, we conclude that $T_n= T$, and that the proposition is proved.
\end{proof}

\medskip

\textbf{Step 2: Passing to the limit} We only sketch the subsequent arguments. 
By the previous estimates, extracting if necessary, we find that $\{\omega_n\}_{n\in \N}$ converges to some $\omega$ in $L^\infty$ weak-$\ast$ on $[0,\widetilde{T}]\times \RR$. Moreover, setting $u=K\ast \omega$, we infer that $\{u_n\}_{n\in \N}$ converges to $u$ locally uniformly on $[0,\widetilde{T}]\times \RR$ (see for instance \cite[Sect. 6.1]{GLS1}). On the other hand, the bounds \eqref{syst:3-dist}-\eqref{ineg:d-31} (or \eqref{ineg:d-2}) imply that each sequence $\{\ddot{h}_{k,n}\}_{n\in \N}$ is uniformly bounded on $[0,\widetilde{T}]$. By Ascoli's theorem, extracting again if necessary, we obtain that each $\{(h_{k,n},\dot{h}_{k,n})\}_{n\in \N}$ converges uniformly to some $(h_k,\dot{h}_k)$ on $[0,\widetilde{T}]$, and passing to the limit in \eqref{syst:3-ODE}, we see that the points $\{h_k\}$ satisfy the desired system of ODE in \eqref{syst:1}. Note that in particular that they satisfy
\begin{equation}
\label{ineg:d-5}
\max_{k}\max_{[0,\widetilde{T}]}|h_k|\leq C,\quad \min_{j\neq k} \min_{ [0,\widetilde{T}]} |h_{j}-h_{k}|\geq \rho>0
\end{equation}
and
\begin{equation}\label{ineg:d-6}
\max_k \max_{ [0,\widetilde{T}]} |\dot{h}_{k}|\leq C.
\end{equation}

Finally, coming back to \eqref{syst:3}, we can pass to the limit exploiting the previous types of convergence to show that $\omega$ is a weak solution of the first PDE in \eqref{syst:1} on $[0,\widetilde{T}]$.

Iterating this construction we reach existence up to the first time of collisions.

If all the circulations have the same sign, we take $T>0$, and we can replace $\widetilde{T}$ by $T$ in all the arguments above since for all $n$ we have a solution $\omega_n$ and $\{h_{k,n}\}$ on $[0,T_n=T]$. This shows that no collision occurs in finite time.

\section{Proof of Theorem~\ref{theorem:main-several-existence-lagrangian}}
\label{sec:lagrangian}

In all this section, $\omega$ denotes any weak solution of \eqref{syst:1-2} on $[0,T]$, where $\{h_k\}$ are given trajectories belonging in $W^{2,\infty}([0,T])$. We assume the analog initial condition as
\eqref{syst:2}:
\begin{equation}
 \label{syst:ini-3}
\omega(0,\cdot)=\omega_0\in L^\infty(\RR)\text{ compactly supported in some } B(0,R_0). 
\end{equation}

We assume that no collision occurs, i.e.
\begin{equation}\label{syst:4-dist}
\min_{t\in [0,T]} \min_{k\neq p} |h_{k}(t)-h_{p}(t)|\geq \rho
\end{equation}
for some $\rho>0$. The purpose of this section is to show Theorem~\ref{theorem:main-several-existence-lagrangian}. We emphasize that the proof does not use the dynamics of $\{h_k\}$.

\subsection{Regular Lagrangian flow}\label{subsec:lagrangian}

We show here that there exists a unique regular Lagrangian flow as defined in Definition~\ref{def:1}.

\medskip

Recall the general following abstract result by Ambrosio \cite[Theorems 3.3 and 3.5]{ambrosio}. Given a vector field $v$ in $L^1_\text{loc}([0,T]\times \R^d)$, if existence and uniqueness for the continuity equation
$$\partial_t \omega+\text{div}(v\omega)=0, \quad \omega(0, \cdot)=\omega_0\in L^1\cap L^\infty$$
 hold in $L^\infty([0,T],L^1\cap L^\infty)$ then the regular Lagrangian flow $X$ for $v$ exists and is unique, and the unique solution is then given by $\omega(t,\cdot)=X(t,\cdot)_\#\omega_0$. 

\medskip
In order to apply this result to the present setting, we introduce the divergence-free field
\begin{equation}\label{def:field}
v(t,x)=u(t,x)+\sum_{j=1}^N \gamma_j K(x-h_j(t)).
\end{equation} 
By Corollary~\ref{coro:appendix} in the Appendix, the transport equation associated to $v$ admits a unique solution (which is renormalized by Lemma~\ref{appendix:1}: for any continuous function $\beta$ growing not too fast at infinity, the function $\beta(\omega)$ is also a solution). Therefore Ambrosio's result yields the existence and uniqueness of the regular Lagrangian flow $X$ associated to $v$, and we have $\omega(t,\cdot)=X(t,\cdot)_\#\omega_0$. This proves the first claim of Theorem~\ref{theorem:main-several-existence-lagrangian}. 

\medskip
Again by Corollary~\ref{coro:appendix}, the renormalization property ensures that 
\begin{equation}\label{eq:norms}\|\omega(t,\cdot)\|_{L^p}=\|\omega_0\|_{L^p}, \quad 1\leq p\leq +\infty.
\end{equation}

We derive first the following property:
\begin{proposition}\label{prop:ineg:flot}
There exists $C$ depending only on $T$, on $\|h_k\|_{L^\infty([0,T])}$ and on the initial data such that 
\begin{equation*}\label{ineg:flot}
\sup_{t\in [0,T]}|X(t,x)|\leq |x|+ C, \quad \text{ for a.e. } x\in \RR.
\end{equation*}
\end{proposition}
 
 \begin{proof} Let us introduce the set of tubular neighborhoods 
 $$\Sigma:=\bigcup_{t\in [0,T]}\bigcup_{k=1}^N B(h_{k}(t),1),$$
 which is bounded set, let say included in a ball $B(0,C_{1})$. Outside $\Sigma$, the velocity $v$ is uniformly bounded by $C_{2}$ (see Proposition~\ref{prop:reg-u}).
 
 Therefore, for any $x\in \mathbb{R}^2$ such that $X(\cdot,x)$ is absolutely continuous on $[0,T]$, the map $t\to X(t,x)$ starts from $x$, has a Lipschitz variation outside $\Sigma$ and 
 can evolve with a diverging velocity inside $\Sigma$, but remaining bounded.
 
Thus, setting $C= C_{1}+C_{2}T$ proves the proposition.
\end{proof}

The following corollary gives the second point in Theorem~\ref{theorem:main-several-existence-lagrangian}.
\begin{corollary}
 \label{coro:support-several} The vorticity $\omega(t,\cdot)$ is compactly supported for 
all $t\in [0,T]$, with $\supp(\omega(t,\cdot))\subset B(0,R_T)$ for some $R_T$ depending only on the initial data, on $\|h_k\|_{L^\infty([0,T])}$ and on $T$.
\end{corollary}

\begin{proof}
We have $\omega(t,\cdot)=X(t,\cdot)_\#\omega_0$ and $\omega_0$ is compactly supported in $B(0,R_0)$, so it follows from Proposition~\ref{prop:ineg:flot} that $\omega(t,\cdot)$ is compactly supported for all $t\in [0,T]$, with $\supp(\omega(t,\cdot))\subset B(0,R_T)$ for $R_T= R_0+C$ (with $C$ given in Proposition~\ref{prop:ineg:flot}).
\end{proof}

\subsection{Vorticity trajectories}

For the third point in Theorem~\ref{theorem:main-several-existence-lagrangian}, we have to show that for almost every $x\in \RR$, we have $X(t,x)\neq h_k(t)$ for all $t\in [0,T]$ and for all $k=1,\ldots,N$. 

\medskip

For almost every $x\in \RR\setminus \cup_j \{ h_{j,0}\}$ such that $X(\cdot,x)$ is an absolutely continuous solution on $[0,T]$ to the ODE with field $v$ defined in \eqref{def:field}, by time continuity, there exists $T^\ast(x)$ such that 
$$\min_j \min_{t\in [0,T^\ast(x)]}|X(t,x)-h_j(t)|>0.$$

We may then consider the local microscopic energies near the points $h_k(t)$ on $[0,T^\ast(x)]$:
\begin{equation}\label{def-Fk}
F_k(t)=\sum_{j=1}^N \frac{\gamma_j}{2\pi} \ln |X(t,x)-h_j(t)|+\varphi_\eps(t,X(t,x))+\langle X(t,x),\dot{h}_k^\perp(t)\rangle,
\end{equation}
where we recall that $\varphi_\eps$ denotes the regularization of the stream function, see \eqref{def:stram-2}.

\medskip

 On the other hand, the result in \cite[Proposition 4.1]{LM} states the continuity of $u$ on $[0,T]\times \RR$. Therefore, the field $v(\cdot,X(\cdot,x))$ is continuous on 
 $[0,T^\ast(x)]$. So we infer that $X(\cdot,x)$ is differentiable on $[0,T^\ast(x)]$ with $\frac{d}{dt}X(t,x)=v(t,X(t,x))$. This enables to perform the following estimate on the local energies.

\begin{proposition}
\label{prop:Fi}
We have for $t\in [0,T^\ast(x)]$ and for all $k=1,\ldots, N$,
\begin{equation*}
|F'_k(t)|\leq C\Big(1+|x|+\eps |X(t,x)-h_k(t)|^{-1}+\sum_{j\neq k}|X(t,x)-h_j(t)|^{-1}\Big).
\end{equation*}
\end{proposition}
In the previous statement, $C$ is independent of $\varepsilon$ whereas $F_k$ depends on $\varepsilon$.

\begin{proof} 
In the subsequent proof we set for clarity:
$$X=X(t,x),\quad u=u(t,X(t,x)), \quad \varphi_\eps=\varphi_\eps(t,X(t,x)),\quad h_k=h_k(t),$$ 
 and we compute on $[0,T^\ast(x)]$
\begin{align*}
F_k'=&\sum_{j=1}^N \left\langle \frac{\gamma_j}{2\pi}\frac{X-h_j}{|X-h_j|^2},\sum_{m=1}^N\gamma_m K(X-h_m)+u-\dot{h}_j\right\rangle \\
&+\partial_t \varphi_\eps+\langle \dot{X},\nabla\varphi_\eps\rangle+\langle \dot{X}, \dot{h}_k^\perp\rangle+\langle X,\ddot{h}_k^\perp\rangle\\
=& \left\langle \sum_{j=1}^N \frac{\gamma_j}{2\pi}\frac{X-h_j}{|X-h_j|^2},\sum_{m=1}^N \frac{\gamma_m}{2\pi}\frac{(X-h_m)^\perp}{|X-h_m|^2} \right\rangle
+\sum_{j=1}^N \left\langle \frac{\gamma_j}{2\pi}\frac{X-h_j}{|X-h_j|^2}, u -\dot{h}_j\right\rangle \\
&+\partial_t \varphi_\eps+\langle \dot{X},\nabla\varphi_\eps\rangle+\left\langle \dot{X}, \dot{h}_k^\perp\rangle+\langle X,\ddot{h}_k^\perp\right\rangle \\
=& \left\langle \frac{\gamma_k}{2\pi}\frac{X-h_k}{|X-h_k|^2},
u-\dot{h}_k\right\rangle 
+\sum_{j\neq k} \left\langle \frac{\gamma_j}{2\pi}\frac{X-h_j}{|X-h_j|^2},
u-\dot{h}_j\right\rangle \\
&+\partial_t \varphi_\eps+\langle \dot{X},\nabla\varphi_\eps\rangle+\left\langle \dot{X}, \dot{h}_k^\perp\right\rangle+\langle X,\ddot{h}_k^\perp\rangle.
\end{align*}
Next, using again that $X$ satisfies the ODE with field $v$ defined in \eqref{def:field}, we have
\begin{equation}\label{eq:flot}
\frac{\gamma_k}{2\pi}\frac{X-h_k}{|X-h_k|^2}=-\dot{X}^\perp+u^\perp-\sum_{j\neq k}\frac{\gamma_j}{2\pi}\frac{X-h_j}{|X-h_j|^2},
\end{equation}
hence we get
\begin{align*}
F'_k=&\langle -\dot{X}^\perp + u^\perp, u-\dot{h}_k\rangle-\sum_{j\neq k}\left\langle \frac{\gamma_j}{2\pi} \frac{X-h_j}{|X-h_j|^2} , u-\dot{h}_k\right\rangle \\
&+\sum_{j\neq k} \left\langle \frac{\gamma_j}{2\pi}\frac{X-h_j}{|X-h_j|^2},
u-\dot{h}_j\right\rangle 
+\partial_t \varphi_\eps+\langle \dot{X},\nabla\varphi_\eps\rangle+\langle \dot{X}, \dot{h}_k^\perp\rangle+\langle X,\ddot{h}_k^\perp\rangle\\
=&\langle \dot{X}^\perp, - u +\nabla^\perp \varphi_\eps \rangle -\langle u^\perp,\dot{h}_k\rangle \\
&+\sum_{j\neq k}\left\langle \frac{\gamma_j}{2\pi} \frac{X-h_j}{|X-h_j|^2} , \dot{h}_k-\dot{h}_j\right\rangle
+\partial_t \varphi_\eps+\langle X,\ddot{h}_k^\perp\rangle.
\end{align*}
Hence, plugging the equality $\nabla^\perp\varphi_\eps=u+R_\eps$, with $R_\eps$ defined in Proposition~\ref{prop:reste-vitesse}, we have
\begin{align*}
F'_k=&\langle \dot{X}^\perp,R_\eps \rangle -\langle u^\perp,\dot{h}_k\rangle +\sum_{j\neq k}\left\langle \frac{\gamma_j}{2\pi} \frac{X-h_j}{|X-h_j|^2} , \dot{h}_k-\dot{h}_j\right\rangle
+\partial_t \varphi_\eps+\langle X,\ddot{h}_k^\perp\rangle.
\end{align*}
By Proposition~\ref{prop:reste-vitesse} together with \eqref{eq:norms} and \eqref{eq:flot}, we have on the one hand
\begin{equation*}
\left|\langle \dot{X}^\perp,R_\eps \rangle\right|\leq C \eps \|\omega_0\|_{L^\infty}\ \left( \sum_{j=1}^N |X-h_j|^{-1}+\|u\|_{L^\infty}\right).
\end{equation*}

On the other hand, as $h_k \in W^{2,\infty}$, we obtain by Proposition~\ref{prop:ineg:flot}
\begin{multline*}
\Big|-\langle u^\perp ,\dot{h}_k\rangle 
+\sum_{j\neq k}\left\langle \frac{\gamma_j}{2\pi} \frac{X-h_j}{|X-h_j|^2} , \dot{h}_k-\dot{h}_j\right\rangle 
 +\langle X,\ddot{h}_k^\perp\rangle\Big|\\
\leq C\Big(\|u\|_{L^\infty}+\sum_{j\neq k}|X-h_j|^{-1}+|x| +1 \Big).
\end{multline*}
Finally, recalling that $\|\partial_t \varphi_\eps\|_{L^\infty}\leq C$ by Proposition~\ref{prop:ineq:partialt-several} and that 
$\|u\|_{L^\infty}\leq C$ by Proposition~\ref{prop:reg-u}, the conclusion follows.
\end{proof}

\begin{corollary}\label{coro:inf-several}For almost every $x$ in $\RR$ we can take $T^\ast(x)=T$, more precisely
$$\min_j \min_{t\in [0,T]}|X(t,x)-h_j(t)|>0,\quad \text{for a.e. } x\in \RR.$$
\end{corollary}

\begin{proof} We argue by contradiction, assuming that $T^\ast(x)=T$ is impossible for some $x\in \RR\setminus\{ h_{k,0}\}$ where the flow exists, so that there exist $k\in \{1,\ldots,N\}$ and $\tilde T<T$ such that $\liminf_{t\to \tilde T} |X(t,x)-h_k(t)|=0$ and $\min_j \min_{t\in [0,T^\ast]}|X(t,x)-h_j(t)|>0$ for any $T^\ast<\tilde T$. We further set $X(t)=X(t,x)$.
Let $t_n\to \tilde T$ such that $|X(t_n)-h_k(t_n)|\to 0$ as $n\to +\infty$. We recall that $\rho$ is defined by \eqref{syst:4-dist}. For $n$ sufficiently large we have 
$|X(t_n)-h_k(t_n)|<\rho/K$, with $K>3$ large to be determined later on.
We take $t'_n$ maximal such that on $[t_n,t'_n)$ we have
 $|X(t)-h_k(t)|<\rho/3$. In particular, by \eqref{syst:4-dist}, for $j\neq k$ we have $|X(t)-h_j(t)|\geq 2\rho/3$ on $[t_n,t'_n)$.
 We assume first that $t'_n< \widetilde{T}$: then we have $|X(t'_n)-h_k(t'_n)|=\rho/3$. We fix $n\in \N$. For all $\eps>0$, 
by the definition of $F_k$, we write
 \begin{equation}\label{esti:Fi}
 \begin{split}
\frac{\gamma_k}{2\pi}& \ln \left(\frac{ |X(t'_n)-h_k(t'_n)|} {|X(t_n)-h_k(t_n)|}\right)
=\sum_{j\neq k}^N \frac{\gamma_j}{2\pi} \ln \left( \frac{|X(t_n)-h_j(t_n)|}{|X(t'_n)-h_j(t'_n)|}\right) +\int_{t_n}^{t'_n}F'_k(\tau)\,d\tau \\
&+\varphi_\eps(t_n,X(t_n))-\varphi_\eps(t'_n,X(t'_n))
+\langle X(t_n),\dot{h}_k^\perp(t_n)\rangle-\langle X(t'_n),\dot{h}_k^\perp(t'_n)\rangle,
 \end{split}
 \end{equation}
 so by \eqref{est:phi}, Proposition~\ref{prop:Fi} and the previous estimates we get
 \begin{equation}\label{est:aout}\begin{split}
\frac{|\gamma_k|}{2\pi} \ln\left(\frac{K}{3}\right)&\leq C\Big(1+|x| + \eps\int_{t_n}^{t'_n}|X(\tau)-h_k(\tau)|^{-1}\,d\tau\Big)
 \end{split}
 \end{equation}
Letting $\eps\to 0$ for fixed $n$, we find
\begin{equation*}
 \ln \left(\frac{K}{3}\right)\leq C,
\end{equation*}
which is a contradiction for $K$ sufficiently large (depending on $x$ and on the initial conditions). So we have $t'_n=\widetilde T$, hence
 \begin{equation}\label{esti:push}
|X(t,x)-h_k(t)|<\frac{\rho}{3}\text { and } |X(t,x)-h_j(t)|\geq \frac{2\rho}{3} \text{ on } [t_n,\widetilde{T}).
\end{equation}
We have therefore localized the fluid trajectory $X(t)$ in the neighborhood of one point vortex trajectory $h_k(t)$, namely we have proved that if the trajectory goes too close to $h_{k}$, it stays in a neighborhood of radius $\rho/3$. 
We fix $n_0\in \mathbb{N}$ sufficiently large so that $|X(t_{n_0})-h_k(t_{n_0})|<\rho/K$. We come back to \eqref{esti:Fi}, replacing $t'_n$ by any $t\in [t_n, \widetilde{T})$, and we apply 
again Proposition~\ref{prop:Fi}:
\begin{equation*}
 \ln |X(t)-h_k(t)|\geq \ln |X(t_{n_0})-h_k(t_{n_0})|-C\Big(1+ |x| +\eps \int_{t_{n_0}}^t |X(\tau)-h_k(\tau)|^{-1}\,d\tau\Big).
 \end{equation*} 
Letting $\eps \to 0$ we find 
\begin{equation*}
|X(t)-h_k(t)|\geq |X(t_{n_0})-h_k(t_{n_0})|e^{-C(1+|x|)}\text { on } [t_{n_0},\widetilde{T}),
\end{equation*}
 which contradicts the fact that $\liminf_{t\to \widetilde{T}} |X(t,x)-h_k(t)|=0$.
 
Hence we conclude that $T^\ast(x)=T$ is possible.
\end{proof}

\begin{corollary}
\label{coro:ODE}
For a.e. $x\in \RR$, the map $ X(\cdot,x)$ is the unique differentiable solution on $[0,T]$ of the ODE $$\frac{d}{dt} \gamma(t)=u(t,\gamma(t))+\sum_{j=1}^N \gamma_j K(\gamma(t)-h_j(t)),\quad \gamma(0)=x,$$such that 
$\min_j \min_{[0,T]} |\gamma(t)-h_j(t)|>0.$
\end{corollary}
\begin{proof}
We gather the already mentioned time continuity of $u$, the log-Lipschitz space regularity for $u$ stated in Proposition~\ref{prop:reg-u}, the no collision property of Corollary~\ref{coro:inf-several}, and the fact that $K$ is Lipschitz away from the origin. Invoking Osgood's Lemma, we can then conclude.
\end{proof}

We finish this paragraph with an additional estimate on the Lagrangian trajectories, which can be derived easily from the proof of Corollary~\ref{coro:inf-several}.

\begin{proposition}\label{prop:dmin-several}
Let $\omega$ be any weak solution of \eqref{syst:1-2} on $[0,T]$ with initial datum \eqref{syst:ini-3},
where $\{h_k\}$ are given trajectories in $W^{2,\infty}([0,T])$ satisfying the no collision property \eqref{syst:4-dist}. 
There exist $0<\delta<\min(\rho/3,1)$, $0<\delta_1<1$ and $0<\delta_2<1$, depending only on $T$, $\|h_k\|_{W^{2,\infty}([0,T])}$, $\rho$, $R_0$ and $\|\omega_0\|_{L^\infty}$,
satisfying the following property: 

\medskip

Let $x\in \supp(\omega_0)$ such that $\min_j \min_{t\in [0,T]} |X(t,x)-h_j(t)|>0$.

If $|X(t_0,x)-h_k(t_0)|<\delta$ for some $t_0\in [0,T]$ and $k\in \{1,\ldots,N\}$, then 
$$\delta_1|x-h_k(0)|\leq |X(t,x)-h_k(t)|<\frac{\rho}{3},\quad \forall t\in [0,T].$$

\medskip

If $\min_j |X(t_0,x)-h_j(t_0)|>\delta$ for some $t_0\in [0,T]$, 
then $$\delta_2\leq \min_j |X(t,x)-h_j(t)|,\quad \forall t\in [0,T].$$

\end{proposition}

\begin{proof}
We start with the first estimate. We come back to the proof of Corollary~\ref{coro:inf-several} above, with $t_n$ replaced by $t_0$, $t'_n$ 
replaced by $T$. With $K>3$ a sufficiently large number to be chosen, we set 
$\delta= \rho/K$. By \eqref{est:aout} we obtain, using that $|x|\leq R_0$ since $x$ belongs to $\supp(\omega_0)$:
\begin{equation*}\begin{split}
\frac{|\gamma_k|}{2\pi} \ln\left(\frac{K}{3}\right)&\leq C\Big(1+\eps\int_{t_0}^{T}|X(\tau)-h_k(\tau)|^{-1}\,d\tau\Big).
 \end{split}
 \end{equation*}
Letting $\eps\to 0$, we find a contradiction if $K$ is sufficiently large (depending only on $T$, $\|h_k\|_{W^{2,\infty}([0,T])}$, $R_0$ and $\|\omega_0\|_{L^\infty}$).
Hence by the same arguments as those leading to \eqref{esti:push} we obtain: 
$$|X(t,x)-h_k(t)|< \frac{\rho}{3},\quad |X(t,x)-h_j(t)|\geq \frac{2\rho}{3}\text{ for }j\neq k \text{ on }[t_0,T].$$ We can invoke the 
same arguments to obtain the estimates above on $[0,t_0]$. 
Therefore, by Proposition~\ref{prop:Fi}, this yields:
$$\lim_{\eps \to 0} \int_0^T|F'_k(\tau)|\,d\tau \leq C,$$
so that, using again \eqref{esti:Fi} with $t_n$ replaced by $0$ and $t'_n$ by $t$, we get
\begin{equation*}
\ln |X(t,x)-h_k(t)|\geq \ln |x-h_k(0)|- C \text { on } [0,T],
\end{equation*}
for a constant $C$, so the first part is proved by setting $\delta_1=e^{-C}$.

\medskip

We turn now to the second part. Let $\widetilde{K}\geq 1$ be a number to be determined later on.
Let $(t_1,t_2)\subset (0,T)$ containing $t_0$ and be maximal such that $\min_j |X(t,x)-h_j(t)|> \delta/\widetilde{K}$ on $(t_1,t_2)$.
If $(t_1,t_2)\neq (0,T)$, let $k\in \{1,\ldots,N\}$ such that
$|X(t_1,x)-h_k(t_1)|=\delta/\widetilde{K}$ (or $|X(t_2,x)-h_k(t_2)|=\delta/\widetilde K$).
Repeating the first part of the proof of Corollary~\ref{coro:inf-several} with $t_n=t_0$ and $t'_n=t_1$ (or $t'_n=t_2$), we find $|\ln \widetilde K|\leq C$ 
which is a contradiction provided $\widetilde K$ is sufficiently large (depending only on $T$, $\|h_k\|_{W^{2,\infty}([0,T])}$, $R_0$ and $\|\omega_0\|_{L^\infty}$). 
So, setting $\delta_2=\delta/\widetilde K$, the conclusion follows.

\end{proof}

\subsection{Decomposition of the vorticity and reduction to the case of one point vortex.}

In all this subsection we assume moreover that $\omega_0$ is constant in a neighborhood of $\{ h_{k}(0) \}$, namely that \eqref{constant} holds.
The purpose here is to show the last property of Theorem~\ref{theorem:main-several-existence-lagrangian}: the vorticity remains constant in a neighborhood of each point vortex.
To this aim, we will first reduce the problem to the case of one single point vortex. In the next subsection, we will then establish the desired property.

\medskip
 
Let $\delta$, $\delta_1$ and $\delta_2$ be the constants introduced in Proposition~\ref{prop:dmin-several}. 
 We decompose $\omega_0$ as
$$\omega_0=\sum_{k=1}^N \omega_{0,k}+\omega_{0,r},$$ where
$$\omega_{0,k}=\omega_0 \mathds{1}_{B(h_{k}(0),\delta)},\quad k=1,\ldots,N$$
and 
$$\omega_{0,r} \quad \text{is supported in }\RR\setminus \cup_{j=1}^N B( h_{j}(0),\delta).$$

By uniqueness of the weak solution to the linear transport equation associated to the field
\begin{equation*}v(t,x)=u(t,x)+\sum_{j=1}^N \gamma_j K(x-h_j(t)),\end{equation*} 
 (see the Appendix and the beginning of Subsection~\ref{subsec:lagrangian}), $\omega$ may then be decomposed as
\begin{equation*}
\omega(t,\cdot)=\sum_{k=1}^N X(t,\cdot)_ \#\omega_{0,k}+X(t,\cdot)_\#\omega_{0,r}=\sum_{k=1}^N \omega_k(t,\cdot)+\omega_{r}(t,\cdot).
\end{equation*}

Let $K_{\delta}=1/(2\pi)\nabla^\perp \ln_{\delta}$, where $\ln_\eps$ is defined in Subsection~\ref{subsec:basic}. So $K_\delta$ is a smooth, divergence-free map coinciding with $K$ on $\RR\setminus B(0,\delta)$ such that $\|K_{\delta}\|_{L^\infty}\leq C \delta^{-1}.$

\medskip

Let $k=1,\ldots,N$. By the first part of Proposition~\ref{prop:dmin-several}, by definition of $\delta$, we have
\begin{equation}\begin{split}\label{ineq:gronwall}
& |{X}(t,x)-h_k(t)|<\frac{\rho}{3}, 
\quad \forall t\in [0,T],\quad \text{ for a.e. }x \in \text{supp}(\omega_{0,k}).
\end{split}
\end{equation}
Therefore,
\begin{equation*}
\min_{j\neq k} |{X}(t,x)-h_j(t)|\geq \frac{2\rho}{3}>\delta,
\quad \forall t\in [0,T]\quad \text{ for a.e. }x \in \text{supp}(\omega_{0,k}).
\end{equation*}
So by Corollary~\ref{coro:ODE}, we have
\begin{equation}\label{eq:flot-tilde}
X(t,x)=\widetilde{X_k}(t,x), \quad \text{ for a.e. }x\in \text{supp}(\omega_{0,k}),\end{equation}
where $\widetilde{X}_k$ is the unique regular Lagrangian flow associated to the field 
\begin{equation}\label{def:drus}
\widetilde v_k(t,x)=u(t,x)+\sum_{j\neq k}\gamma_j K_{\delta}(x-h_j(t))+\gamma_k K(x-h_k(t)).
\end{equation}
In particular,
\begin{equation}\label{eq:carre-0}
\omega_k(t,\cdot)=X(t,\cdot)_\#\omega_{0,k}=\widetilde{X}_k(t,\cdot)_\#\omega_{0,k}.
\end{equation}

We observe here for later use that the same argument applied to $\omega^2$ (noting that it is also a distributional solution of \eqref{syst:1-2} with initial datum $\omega_0^2$) yields 
\begin{equation}\label{eq:carre}\omega_k^2(t,\cdot)=X(t,\cdot)_\#\omega_{0,k}^2=\widetilde{X_k}(t,\cdot)_\#\omega_{0,k}^2,\quad k=1,\ldots,N.
\end{equation}
 So we are left with the case of a linear transport equation with field $\widetilde v_k$ given by the superposition \eqref{def:drus} of a regular part 
\begin{equation}\label{def:push}
u_k(t,x)=u(t,x)+\sum_{j\neq k}\gamma_j K_{\delta}(x-h_j(t))\end{equation} and a singular part generared by only one point vortex:
$$\gamma_k K(x-h_k(t)).$$ The analysis of this case was performed in \cite{bresiliens-miot}.
It was proved in particular that for all $t\in [0,T]$, the regular Lagrangian flow $\widetilde{X}_k$ associated to $\widetilde v_k$
 is the limit in $L^1_{\loc}(\RR)$ of the sequence $\widetilde{X_{k,n}}(t,\cdot)$, where $\widetilde{X_{k,n}}$ is the flow 
associated to any regularization of $\widetilde v_k$:
$$\widetilde v_{k,n}(t,x)=u_{k,n}(t,x)+\frac{\gamma_k}{2\pi}\frac{(x-h_k(t))^\perp}{|x-h_k(t)|^2+n^{-2}},$$ with $u_{k,n}$ a smooth and divergence-free approximation of $u_k$. By Liouville's theorem, $\widetilde{X_{k,n}}(t,\cdot)$ thus preserves Lebesgue's measure. 
Moreover, Proposition~\ref{prop:ineg:flot} also applies to $\widetilde{X_{k,n}}$ (with a constant independent of $n$). Therefore, passing to the limit, we conclude that $\widetilde{X_k}(t,\cdot)$ preserves Lebesgue's measure:
\begin{equation}
\label{prop:lebesgue}
\widetilde{X_k}(t,\cdot)_\#dx=dx,\quad \forall t\in [0,T].
\end{equation}

\medskip

We next derive a localization property for $\widetilde{X_k}(t,\cdot)$.

\begin{proposition}\label{prop:loc}
For all $R>0$, there exists $C_R$ depending only on $R$, $T$, $\|h_k\|_{W^{2,\infty}([0,T])}$, $\rho$, $R_0$ and $\|\omega_0\|_{L^\infty}$, but not on $\delta_0$, such that
\begin{multline*}
C_R|x-h_k(0)|\leq |\widetilde{X_k}(t,x)-h_k(t)|,\\ 
\forall t\in [0,T],\quad \text{ for all }x\in B(0,R)\setminus\{h_{k}(0)\},\quad k=1,\ldots,N.
\end{multline*}
\end{proposition}

\begin{remark}
By \eqref{eq:flot-tilde} and by Proposition~\ref{prop:dmin-several}, we already know that this holds for a.e. $x$ in $\text{supp}(\omega_{0,k})$.
\end{remark}

\begin{proof} 
As long as $\widetilde{X_k}(t,x)\neq h_k(t)$, we introduce the new energy
\begin{equation*}
\widetilde{F_k}(t)=\frac{\gamma_k}{2\pi}\ln |\widetilde{X_k}(t,x)-h_k(t)|+\varphi_\eps(t,\widetilde{X_k}(t,x))+\psi_{\delta}(t,\widetilde{X_k}(t,x))+\langle \widetilde{X_k}(t,x), \dot{h}_k(t)^\perp\rangle,
\end{equation*}
where
$$ \psi_{\delta}(t,x)=\sum_{j\neq k}\frac{\gamma_j}{2\pi} \ln_{\delta}|x-h_j(t)|,$$
so that
 $$\nabla^\perp \psi_{\delta}(t,x)=\sum_{j\neq k}\gamma_j K_{\delta}(x-h_j(t)).$$
Exactly as in the proof of Proposition~\ref{prop:Fi}, we compute, recalling the definition \eqref{def:push} of $u_k$,
\begin{align*}
\widetilde{F_k}'(t)=& -\left\langle u_k^\perp(t,\widetilde{X_k}(t,x)),\dot{h}_k(t)\right\rangle
+\partial_t \varphi_\eps(t,\widetilde{X_k}(t,x))+\partial_t \psi_{\delta}(t,\widetilde{X_k}(t,x))\\
&+\left\langle \dot{\widetilde{X_k}}(t,x)^\perp,R_\eps(t,\widetilde{X_k}(t,x))\right\rangle
+\left\langle \widetilde{X_k}(t,x),\ddot{h}_k^\perp(t)\right\rangle.
\end{align*}
Using the uniform bounds on $u_k$, $\partial_t \psi_{\delta}$, $h_j$, $\dot{h}_j$ and $\ddot{h}_j$ for $j=1,\ldots,N$, and using the previous bounds for $\partial_t \varphi_\eps$ and $R_\eps$ we therefore get for all $\eps>0$
\begin{equation*}
|\widetilde{F_k}'(t)| \leq C\big(|x|+\eps |\widetilde{X_{k}}(t,x)-h_k(t)|^{-1}+1\big).
\end{equation*}
We may now conclude exactly as in the proof of the first part of Proposition~\ref{prop:dmin-several}: as long as $\widetilde{X_k}(t,x)\neq h_k(t)$, letting $\eps$ tend to zero after integrating the inequality above on $[0,t]$, we get
\begin{equation*}
\left|\ln \left(\frac{|\widetilde{X_k}(t,x)-h_k(t)|}{|x-h_{k,0}|}\right)\right|\leq C(1+|x|),\end{equation*}
where $C$ depends on $\delta$, $T$, $\|h_k\|_{W^{2,\infty}([0,T])}$, $\rho$ and on the initial data. So, setting 
$C_R=e^{-C(1+R)}$, the conclusion follows.
\end{proof}

\begin{proposition}\label{prop:loca-1}
We have
$$|x-h_{k}(0)|\leq C( |\widetilde{X_k}(t,x)-h_k(t)|+1),$$
where $C$ depends only on $T$, $\|h_k\|_{W^{2,\infty}([0,T])}$, $\rho$, $R_0$ and $\|\omega_0\|_{L^\infty}$, but not on $\delta_0$.
\end{proposition}
\begin{proof} 
\begin{align*}
\frac{d}{dt} |\widetilde{X_k}&(t,x)-h_k(t)|^2\\
&=2\left\langle \widetilde{X_k}(t,x)-h_k(t), u_k(t,\widetilde{X_k}(t,x))+\gamma_k K(\widetilde{X_k}(t,x)-h_k(t))-\dot{h}_k(t)\right\rangle\\
&=2\left\langle \widetilde{X_k}(t,x)-h_k(t), u_k(t,\widetilde{X_k}(t,x))-\dot{h}_k(t)\right\rangle\\
&\geq -2|\widetilde{X_k}(t,x)-h_k(t)|(\|\dot{h}_k\|_{L^\infty}+\|u_k\|_{L^\infty}),
\end{align*}
hence 
\begin{equation*}
\frac{d}{dt} |\widetilde{X_k}(t,x)-h_k(t)| \geq -C,
\end{equation*}
so the conclusion follows.

\end{proof}

\subsection{The vorticity remains constant in the neighborhood of the point vortices}
We finally establish that the vorticity remains constant in a neighborhood of the point vortices.
Let $C$ be the constant of Proposition~\ref{prop:loca-1}. We set 
$$R=2C+|h_{k}(0)|,$$
and we consider the corresponding constant $C_R$ of Proposition~\ref{prop:loc}.
We may decrease $\delta_0$ so that
$$\max(\delta_0,C_R \delta_0)<\min\left(\delta,\delta_2,1,\frac{2\rho}{3}\right),$$
where we recall $\delta$ and $\delta_2$ were found in Proposition~\ref{prop:dmin-several}.

We fix $t\in [0,T]$.
We claim that
\begin{equation}\label{claim:vanish}
 \omega_j(t,y)=0,\quad \text{for a.e. } y\in B(h_k(t),C_R\delta_0),\quad \forall j\neq k.\end{equation}
Indeed, by \eqref{eq:carre}, considering the $L^1$ function $\varphi_k=\mathds{1}_{B(h_{k}(t),C_R\delta_0)},$ we find 
\begin{equation*}
 \int_{\RR}\omega_j^2(t,y)\varphi_k(y)\,dy
 =\int_{\RR}\omega_{0,j}^2(x)\varphi_k({X}(t,x))\,dx.
 \end{equation*}
On the other hand, for $x\in \text{supp}(\omega_{0,j})$, we have by \eqref{ineq:gronwall} $|{X}(t,x)-h_j(t)|<\rho/3$ and therefore $|{X}(t,x)-h_k(t)|>2\rho /3>C_R\delta_0$. So the right hand side above vanishes, which establishes \eqref{claim:vanish}.

Using that $\delta_2>C_R\delta_0$, by the same arguments as above, the second part of Proposition~\ref{prop:dmin-several} yields that
\begin{equation}\label{claim:vanish-2}
 \omega_r(t,y)=0,\quad \text{for a.e. } y\in \bigcup_{k=1}^N B(h_k(t),C_R\delta_0).\end{equation}

\medskip

Finally, we show that
\begin{equation}\label{claim:vanish-3}
 \omega_k(t,y)=\alpha_k,\quad \text{for a.e. } y\in B(h_k(t),C_R\delta_0).\end{equation}

Indeed, since $\omega_k(t,\cdot)=\widetilde{X_k}(t,\cdot)_\#\omega_{0,k}$, $\omega_k^2(t,\cdot)=\widetilde{X_k}(t,\cdot)_\#\omega_{0,k}^2$, and $\widetilde{X_k}(t,\cdot)_\#dx=dx$ by \eqref{eq:carre-0}, \eqref{eq:carre} and \eqref{prop:lebesgue}, we compute 
\begin{align*}
 \int_{\RR}(\omega_k(t,y)-\alpha_k)^2\varphi_k(y)\,dy \hspace{-2cm}&\\
=&\int_{\RR}\omega_k(t,y)^2\varphi_k(y)\,dy-2\alpha_k \int_{\RR}\omega_k(t,y)\varphi_k(y)\,dy+\alpha_k^2\int_{\RR}\varphi_k(y)\,dy\\
 =&\int_{\RR}\omega_{0,k}^2(x)\varphi_k(\widetilde{X_k}(t,x))\,dx-2\alpha_k \int_{\RR}\omega_{0,k}(x)\varphi_k(\widetilde{X_k}(t,x))\,dx\\
&+\alpha_k^2\int_{\RR}\varphi_{k}(\widetilde{X_k}(t,x))\,dx\\
=&\int_{\RR}(\omega_{0,k}(x)-\alpha_k)^2\varphi_{k}(\widetilde{X_k}(t,x))\,dx\\
=&\int_{\widetilde{X_k}(t,\cdot)^{-1}(B(h_k(t),C_R\delta_0))}(\omega_{0,k}(x)-\alpha_k)^2\,dx.
 \end{align*}
Now we observe that since $C_R\delta_0<1$, by Proposition~\ref{prop:loca-1}, we get 
$$\widetilde{X_k}(t,\cdot)^{-1}\Big(B(h_k(t),C_R\delta_0) \Big)\subset B(h_{k}(0),2C)\subset B(0,R).$$ 
Thus, we are allowed to use Proposition~\ref{prop:loc}, and we have for $x\in \widetilde{X_k}(t,\cdot)^{-1}(B(h_k(t),C_R\delta_0))$: 
$$C_R|x-h_{k}(0)|\leq |\widetilde{X_k}(t,x)-h_k(t)|\leq C_R\delta_0.$$
We get therefore
\begin{equation*}
\int_{\RR}(\omega(t,y)-\alpha_k)^2\varphi_k(y)\,dy
 \leq \int_{B(h_{k}(0),\delta_0)}(\omega_0(x)-\alpha_k)^2\,dx=0,
 \end{equation*}
 and the conclusion follows.

\medskip
In view of \eqref{claim:vanish}, \eqref{claim:vanish-2} and \eqref{claim:vanish-3}, we finally conclude that
\begin{equation}
\label{om:constant}
\omega(t,\cdot)=\alpha_k,\quad \text{a.e. on } B(h_k(t), C_R\delta_0).
\end{equation}

\section{Proof of Theorem~\ref{theorem:main-several}.}

\label{sec:uniqueness-final}

\textbf{Step 1: uniqueness in the case of one point vortex.}

We start with the case $N=1$.
Let $(\omega,h)$ and $(\widetilde{\omega},\widetilde{h})$ two solutions of \eqref{syst:1} with initial datum $(\omega_0,h_0,\ell_0)$ satisfying 
the assumption of Theorem~\ref{theorem:main-several}. So Theorem~\ref{theorem:main-several-existence-lagrangian} holds for both solutions: $\omega$ and $\widetilde{\omega}$ remain constant in a neighborhood of the trajectories of
$h$ and $\tilde{h}$.

 Noting that $u-\widetilde{u}=K\ast(\omega-\widetilde{\omega})$ with $\int (\omega-\widetilde{\omega})=0$ and $\omega,\widetilde{\omega}$ compactly supported, we have $u-\widetilde{u}\in L^2(\R^2)$ (see \cite[Proposition 3.3]{majda-bertozzi}) and we may consider the quantity
\begin{equation*}
D(t)=\|u(t,\cdot)-\widetilde{u}(t,\cdot)\|_{L^2}^2+|h(t)-\widetilde{h}(t)|^2+|\dot{h}(t)-\dot{\widetilde{h}}(t)|^2,\quad t\in [0,T].
\end{equation*}
In what follows we establish a Gronwall inequality for $D(t)$.

We remark that the only difference between \eqref{syst:1} and the vortex-wave system \eqref{syst:wv} is the ODE for the point vortex, 
since the PDE for the vorticity is the same. Thus we may directly use the estimates derived for \eqref{syst:wv} in \cite[Subsection 3.4]{LM} for the quantity $\|u(t,\cdot)-\widetilde{u}(t,\cdot)\|_{L^2}^2$. More precisely, by the estimate (3.9) in \cite{LM} we have for $t\in [0,T^\ast)$ and for all $p\geq 2$
\begin{equation*}
\|u(t,\cdot)-\widetilde{u}(t,\cdot)\|_{L^2}^2\leq C\int_0^t \left( r(\tau)+\sqrt{r(\tau)}f(\sqrt{r(\tau)})+p\:r(\tau)^{1-1/p}\right)\,d\tau,
\end{equation*}
where $$r(t)=\|u(t,\cdot)-\widetilde{u}(t,\cdot)\|_{L^2}^2+|h(t)-\widetilde{h}(t)|^2,$$ and where 
$$f(\tau)=\tau |\ln \tau|.$$ Here, $T^\ast\in [0,T]$ is the largest time such that 
$|h(t)-\widetilde{h}(t)|<\min(1,\delta/2)$ on $[0,T^\ast)$. So using that $r(t)\leq D(t)$, and the inequalities
$\tau f(\tau)\leq f(\tau^2)$, $\tau\leq f(\tau)$ for $\tau\leq 1$ and $f(\tau)\leq p\tau ^{1-1/p}$ (for all $p\geq 2$), we get for $t\in [0,T^\ast)$ and for all $p\geq 2$
\begin{equation}
\label{ineq:vitesse}
\|u(t,\cdot)-\widetilde{u}(t,\cdot)\|_{L^2}^2\leq C\,p\int_0^t D(\tau)^{1-1/p}\,d\tau.
\end{equation}
We emphasize that the property obtained in Theorem~\ref{theorem:main-several-existence-lagrangian} is crucial in order to obtain the previous estimate, 
by implying in particular that $u-\widetilde{u}$ is harmonic in the neighborhood of $h$ and $\tilde{h}$.

We turn next to the estimate for the point vortices. We compute
\begin{align*}
\frac{d}{dt}|h-\widetilde{h}|^2&+\frac{d}{dt}|\dot{h}-\dot{\widetilde{h}}|^2\\
=&2\langle h-\widetilde{h}, \dot{h}-\dot{\widetilde{h}}\rangle
-2\frac{\gamma}{m}\langle \dot{h}-\dot{\widetilde{h}},u(t,h)^\perp-\widetilde{u}(t,\widetilde{h})^\perp\rangle\\
\leq& D(t)+2\frac{\gamma}{m}\sqrt{D(t)}|u(t,h)-u(t,\widetilde{h})|+2\frac{\gamma}{m}\sqrt{D(t)}|(u-\widetilde{u})(t,\widetilde{h})|.
\end{align*}
On the one hand, since $u$ is log-Lipschitz we have $|u(t,h)-u(t,\widetilde{h})|\leq C f(|h-\tilde{h}|)\leq Cf(\sqrt{D(t)})$. On the other hand, 
exactly as in Step 2 in the proof of \cite[Proposition 3.10]{LM}, we rely on \cite[Lemma 3.9]{LM}: using the analyticity of $u-\widetilde{u}$ near $h$ and $\widetilde{h}$, 
that lemma enables to obtain
$$|u(t,h)-\widetilde{u}(t,h)|\leq C \|u(t,\cdot)-\widetilde{u}(t,\cdot)\|_{L^2}.$$
Hence we get finally that for all $p\geq 2$,
\begin{equation}\label{ineq:vortex}
\frac{d}{dt}|h-\widetilde{h}|^2+\frac{d}{dt}|\dot{h}-\dot{\widetilde{h}}|^2
\leq C f(D(t))\leq CpD(t)^{1-1/p},\quad \forall t\in [0,T^\ast).
\end{equation}

Finally, gathering \eqref{ineq:vitesse} and \eqref{ineq:vortex}, we find 
$$D(t)\leq C\,p \int_0^t D(\tau)^{1-1/p}\,d\tau,\quad \forall p\geq 2.$$ So we conclude by usual arguments (see \cite[Chapter 8]{majda-bertozzi} that $D\equiv 0$ on $[0,T^\ast)$. Thus by definition of $T^\ast$ we get $T^\ast=T$ and uniqueness follows on $[0,T]$.

\medskip

\textbf{Step 2: Proof of Theorem~\ref{theorem:main-several} completed}

Once the case of one point is settled, the conclusion of Theorem~\ref{theorem:main-several} follows easily by adapting the proof above to the case of several points, using \eqref{om:constant}, \eqref{ineg:d-5} and \eqref{ineg:d-6}. We refer also to the proof of uniqueness in \cite[Theorem 2.1, Chapter 2]{MiotThesis} dealing with several points.

\label{sec:final-proof}

\section{Some additional properties}

We prove in this section some additional properties for System~\eqref{syst:1} in the case where the circulations and the vorticity have positive sign.

\begin{proposition}\label{prop:energy-cst}
Let $\omega_0$ and $(\{h_{k,0}\},\{\ell_{k,0}\})$ be as in \eqref{syst:2} and let $(\omega,\{h_k\})$ be any corresponding weak solution to \eqref{syst:1} on $[0,T]$. The following quantities are conserved:
\begin{itemize}
\item The energy,
\begin{align*}
\mathcal{H}_0=&\frac{1}{2\pi}\int_{\RR} \int_{\RR}\ln |x-y|\omega(t,y)\omega(t,x)\,dx\,dy+\frac{1}{\pi}\sum_{k=1}^N \gamma_k \int_{\RR}\ln |x-h_k(t)|\omega(t,x)\,dx\\
&+\sum_{j\neq k}\frac{\gamma_k \gamma_j}{2\pi}\ln |h_k(t)-h_j(t)|-\sum_{k=1}^N m_k |\dot{h}_k(t)|^2.
\end{align*}
\item The momentum,
\begin{equation*}
\mathcal{I}_0=\int_{\RR} |x|^2\omega(t,x)\,dx+\sum_{k=1}^N \gamma_k |h_k(t)|^2-2\sum_{k=1}^N m_k h_k(t)^\perp \cdot \dot{h}_k(t).
\end{equation*}
\end{itemize}

\end{proposition}

\begin{proof}(sketch)
For $\eps<\frac{1}{3}\min_{j\neq k}\min_{t\in [0,T]}|h_j(t)-h_k(t)|$, we replace $\ln$ by the smooth function $\ln_\eps$ defined in the first section and we set $\varphi_\eps=\frac1{2\pi}\ln_\eps\ast \omega$ as in \eqref{def:stram-2}, so that, setting
\begin{align*}
\mathcal{H}_\eps=&\int_{\RR} \varphi_\eps(t,x)\omega(t,x)\,dx+2\sum_{k=1}^N \gamma_k \varphi_\eps(t,h_k(t))\\
&+\sum_{j\neq k}\frac{\gamma_j \gamma_k}{2\pi}\ln_\eps |h_j(t)-h_k(t)|-\sum_{k=1}^N m_k |\dot{h}_k(t)|^2,
\end{align*}
we have $\sup_{ [0,T]}|\mathcal{H}_0-\mathcal{H}_\eps|\leq C\eps$, with the quantity $C$ depending only on $\|\omega\|_{L^\infty}$, $\|h_k\|_{L^\infty}$, $m_k$, $\gamma_k$ etc.

It suffices then to compute the time derivative of $\mathcal{H}_\eps$ using the weak formulation for $\omega$ and the ODE for the $h_k's$, which yields
$\sup_{[0,T]}|\dot{\mathcal{H}}_\eps|\leq C\eps$. Letting $\eps$ tend to zero, the conclusion follows.

For $\mathcal{I}_0$ we compute directly the time derivative using the weak formulation for $\omega$ and the ODE for the $h_k's$ and we show that it vanishes, which yields the result.
\end{proof}

With these conservations, we can prove that the massive point vortices are confined if $\omega$ and $\{\gamma_{k}\}$ have the same sign.
\begin{corollary}\label{coro:dmin}
Assume moreover that $$\omega_0\geq 0,\text{ a.e. on }\RR,\quad \gamma_k>0,\: k=1,\ldots,N.$$
Let $(\omega,\{h_k\})$ be any corresponding weak solution to \eqref{syst:1} on $[0,T]$. Then
there exists $C>0$ and $d>0$, depending only on $\mathcal{H}_0$, $\mathcal{I}_0$, $m_k$, $\gamma_k$ and $\|\omega_0\|_{L^\infty}$ and $R_0$, but not on $T$, such that
\begin{equation*}
\sup_{t\in [0,T]}\left(|\dot{h}_k(t)|^2+|{h}_k(t)|^2\right)\leq C
\end{equation*}
and
\begin{equation*}
\inf_{t\in [0,T]} \min_{j\neq k}|h_j(t)-h_k(t)|\geq d.
\end{equation*} 
\end{corollary}

\begin{proof}
Since $\omega$ is transported by the flow, we have $\omega(t,\cdot)\geq 0$ almost everywhere for $t\in [0,T]$. Similarly to the proof of Proposition~\ref{prop:n}, picking $m\neq n$, we have, using that $\ln(|x-y|)\leq |x|+|y|$,
\begin{multline*}
\frac{\gamma_m \gamma_n}{2\pi}\ln |h_m(t)-h_n(t)|\geq \mathcal{H}_0-\frac{1}{2\pi}\int_{\RR}\int_{\RR}
\big(|x|+|y|\big)\omega(t,y)\omega(t,x)\,dx\,dy\\
-\frac{1}{\pi}\sum_{k=1}^N\gamma_k\int_{\RR}\big(|x|+|h_k(t)|\big)\omega(t,x)\,dx
-\sum_{j\neq k}\frac{\gamma_k \gamma_j}{2\pi}\big( |h_k(t)|+|h_j(t)|\big)
\end{multline*}
therefore by Cauchy-Schwartz inequality:
\begin{equation}\label{ineq:appendix-1}
\frac{\gamma_m \gamma_n}{2\pi}\ln |h_m(t)-h_n(t)|
 \geq \mathcal{H}_0-C\left(\int_{\RR} |x|^2\omega(t,x)\,dx+\sum_{k=1}^N \gamma_k |h_k(t)|^2\right)^{1/2},
\end{equation}
where $C$ depends only on $\|\omega(t,\cdot)\|_{L^1}=\|\omega_0\|_{L^1}$ and on $\gamma_k$.

By the same estimates we also obtain
\begin{equation}\label{ineq:appendix-2}
\sum_{k=1}^N m_k |\dot{h}_k(t)|^2\leq - \mathcal{H}_0+C\left(\int_{\RR} |x|^2\omega(t,x)\,dx+\sum_{k=1}^N \gamma_k |h_k(t)|^2\right)^{1/2}.
\end{equation}

On the other hand, by Cauchy-Schwartz inequality:
\begin{align*}
&\int_{\RR} |x|^2\omega(t,x)\,dx+\sum_{k=1}^N \gamma_k |h_k(t)|^2
\leq \mathcal{I}_0+C\left(\sum_{k=1}^N m_k |\dot{h}_k(t)|^2\right)^{1/2}\left(\sum_{k=1}^N \gamma_k |{h}_k(t)|^2\right)^{1/2}\\
\leq& \mathcal{I}_0 +C\left(-\mathcal{H}_0+C\left(\int_{\RR} |x|^2\omega(t,x)\,dx+\sum_{k=1}^N \gamma_k |h_k(t)|^2\right)^{1/2}\right)^{1/2}\left(\sum_{k=1}^N \gamma_k |{h}_k(t)|^2\right)^{1/2}\\
\leq& \mathcal{I}_0 +C\left(1+\int_{\RR} |x|^2\omega(t,x)\,dx+\sum_{k=1}^N \gamma_k |h_k(t)|^2\right)^{3/4}
\end{align*}
where we have used \eqref{ineq:appendix-2}. We conclude that
\begin{equation*}
\int_{\RR} |x|^2\omega(t,x)\,dx+\sum_{k=1}^N \gamma_k |h_k(t)|^2\leq C
\end{equation*}
with $C$ depending only $\mathcal{I}_0, \mathcal{H}_0$ and $\|\omega_0\|_{L^1}.$ Coming back to \eqref{ineq:appendix-1} and \eqref{ineq:appendix-2}, the conclusion follows.
\end{proof}

\appendix
 \section{Some results included in \cite{MiotThesis}}

In this Appendix we gather several results from \cite[Chapter 1]{MiotThesis}. Since that reference is in french we provide here the statements in english and refer to \cite{MiotThesis} for the proofs. Similar results and proofs in the case of one point vortex are also to be found in \cite{LM}.

\begin{lemma}
 \label{appendix:1}Let $\{h_k\}$ be $N$ Lipschitz trajectories on $[0,T]$ without collisions: 
\begin{equation*}
\min_{t\in [0,T]} \min_{k\neq p} |h_{k}(t)-h_{p}(t)|\geq \rho
\end{equation*}
for some $\rho>0$.
Let $\omega$ be a weak solution of the PDE
\begin{equation}
\label{eq:appendix}
\partial_t \omega+\div(v \omega)=0\quad \text{on }[0,T],
\end{equation}where $v$ is the divergence-free velocity field given by
\begin{equation}\label{def:field-appendix}
v(t,x)=u(t,x)+\sum_{j=1}^N \gamma_j K(x-h_j(t)).
\end{equation} 
with $u$ a divergence-free vector field satisfying
\begin{equation}\label{regu:field-appendix}
u\in L^\infty([0,T]\times \RR)\quad \text{ and }u(t,\cdot)\text{ is log-Lipschitz uniformly in time}.
\end{equation}
Let $\beta:\R \rightarrow
\R$ be $C^1$ such that
\begin{equation*}
|\beta'(z)|\leq C(1+ |z|^p),\qquad \forall z\in \R,
\end{equation*}
for some $p\geq 0$. Then for all test function $\psi \in
C_c^\infty([0,T] \times \RR)$, we have
\begin{equation*}
\frac{d}{dt}\int_{\RR} \psi \beta(\omega)\,dx=\int_{\RR} \beta(\omega)
(\partial_t \psi +v\cdot \nabla \psi)\,dx \:\: \mathrm{in }\:
L^1([0,T]).
\end{equation*}
\end{lemma}

This lemma is stated in \cite[Chapter 1, Lemme 1.5]{MiotThesis} in the case where $(\omega, \{h_k\})$ is a weak solution of the vortex-wave system. However a straightforward adaptation of the proof shows that this holds for the linear transport equation \eqref{eq:appendix} with any vector field $v$ given by the decomposition \eqref{def:field-appendix}, where $u$ satisfies the regularity properties \eqref{regu:field-appendix} and where the $h_j'$ are Lipschitz continuous on $[0,T]$ and do not intersect. We emphasize that their precise dynamics is not used to show the renormalization property. 

\medskip

As a consequence of Lemma~\ref{appendix:1} it is observed in \cite[Chapter 1, Remarque 1.3]{MiotThesis} (or in \cite[Lemma 3.2]{LM} for the case of one point) that 
\begin{corollary}\label{coro:appendix} Under the same assumption as in Lemma~\ref{appendix:1} for the $\{h_k\}$, let $\omega$ be a weak solution of the 
PDE \eqref{eq:appendix}. Then for all $1\leq p \leq +\infty$ we have $\|\omega(t,\cdot)\|_{L^p}=\|\omega(0,\cdot)\|_{L^p}.$ In particular, uniqueness of the weak 
solution holds.
\end{corollary}


\begin{thebibliography}{10}

\bibitem{ambrosio}
L.~Ambrosio.
\newblock Transport equation and {C}auchy problem for non-smooth vector fields.
\newblock In {\em Calculus of variations and nonlinear partial differential
  equations}, volume 1927 of {\em Lecture Notes in Math.}, pages 1--41.
  Springer, Berlin, 2008.

\bibitem{ambrosio-survey}
L.~Ambrosio.
\newblock Well posedness of {ODE}'s and continuity equations with nonsmooth
  vector fields, and applications.
\newblock {\em Rev. Mat. Complut.}, 30(3):427--450, 2017.

\bibitem{crippa-delellis}
G.~Crippa and C.~De~Lellis.
\newblock Estimates and regularity results for the {D}i{P}erna-{L}ions flow.
\newblock {\em J. Reine Angew. Math.}, 616:15--46, 2008.

\bibitem{bresiliens-miot}
G.~Crippa, M.~C. Lopes~Filho, E.~Miot, and H.~J. Nussenzveig~Lopes.
\newblock Flows of vector fields with point singularities and the vortex-wave
  system.
\newblock {\em Discrete Contin. Dyn. Syst.}, 36(5):2405--2417, 2016.

\bibitem{DeLellis}
C.~De~Lellis.
\newblock Ordinary differential equations with rough coefficients and the
  renormalization theorem of {A}mbrosio [after {A}mbrosio, {D}i{P}erna,
  {L}ions].
\newblock {\em Ast\'erisque}, (317):Exp. No. 972, viii, 175--203, 2008.
\newblock S\'eminaire Bourbaki. Vol. 2006/2007.

\bibitem{dip-lions}
R.~J. DiPerna and P.-L. Lions.
\newblock Ordinary differential equations, transport theory and {S}obolev
  spaces.
\newblock {\em Invent. Math.}, 98(3):511--547, 1989.

\bibitem{GLS1}
O.~Glass, C.~Lacave, and F.~Sueur.
\newblock On the motion of a small body immersed in a two-dimensional
  incompressible perfect fluid.
\newblock {\em Bull. Soc. Math. France}, 142(3):489--536, 2014.

\bibitem{GLS2}
O.~Glass, C.~Lacave, and F.~Sueur.
\newblock On the motion of a small light body immersed in a two dimensional
  incompressible perfect fluid with vorticity.
\newblock {\em Comm. Math. Phys.}, 341(3):1015--1065, 2016.

\bibitem{LM}
C.~Lacave and E.~Miot.
\newblock Uniqueness for the vortex-wave system when the vorticity is constant
  near the point vortex.
\newblock {\em SIAM J. Math. Anal.}, 41(3):1138--1163, 2009.

\bibitem{Lamb}
H.~Lamb.
\newblock {\em Hydrodynamics}.
\newblock Cambridge university press, 1993.

\bibitem{majda-bertozzi}
A.~J. Majda and A.~L. Bertozzi.
\newblock {\em Vorticity and incompressible flow}, volume~27 of {\em Cambridge
  Texts in Applied Mathematics}.
\newblock Cambridge University Press, Cambridge, 2002.

\bibitem{Marchioro}
C.~Marchioro.
\newblock On the {E}uler equations with a singular external velocity field.
\newblock {\em Rend. Sem. Mat. Univ. Padova}, 84:61--69 (1991), 1990.

\bibitem{MarPul1}
C.~Marchioro and M.~Pulvirenti.
\newblock On the vortex-wave system.
\newblock In {\em Mechanics, analysis and geometry: 200 years after
  {L}agrange}, North-Holland Delta Ser., pages 79--95. North-Holland,
  Amsterdam, 1991.

\bibitem{MarPul}
C.~Marchioro and M.~Pulvirenti.
\newblock {\em Mathematical theory of incompressible nonviscous fluids},
  volume~96 of {\em Applied Mathematical Sciences}.
\newblock Springer-Verlag, New York, 1994.

\bibitem{MiotThesis}
E.~Miot.
\newblock {\em Quelques probl\`emes relatifs \`a la dynamique des points vortex
  dans les \'equations d'Euler et de Ginzburg-Landau complexe}.
\newblock Theses, {Universit{\'e} Pierre et Marie Curie - Paris VI}, Dec. 2009.

\bibitem{MoussaSueur}
A.~Moussa and F.~Sueur.
\newblock On a {V}lasov-{E}uler system for 2{D} sprays with gyroscopic effects.
\newblock {\em Asymptot. Anal.}, 81(1):53--91, 2013.

\bibitem{Toan}
T.~T. Nguyen and T.~T. Nguyen.
\newblock The inviscid limit of navier-stokes equations for vortex-wave data on
  $\mathbb{R}^2$.
\newblock {\em SIAM J. Math. Anal.}, 2019.
\newblock To appear.

\bibitem{Schochet}
S.~Schochet.
\newblock The point-vortex method for periodic weak solutions of the 2-{D}
  {E}uler equations.
\newblock {\em Comm. Pure Appl. Math.}, 49(9):911--965, 1996.

\bibitem{stein}
E.~M. Stein.
\newblock {\em Singular integrals and differentiability properties of
  functions}.
\newblock Princeton Mathematical Series, No. 30. Princeton University Press,
  Princeton, N.J., 1970.

\bibitem{Thomson}
L.~M.~M. Thomson.
\newblock {\em Theoretical Hydrodynamics}.
\newblock London, 1955.

\end{thebibliography}
\def\cprime{$'$}

\end{document}